\documentclass[a4paper,12pt]{article}

\usepackage{amsmath}
\usepackage{amsfonts}
\usepackage{amssymb}
\usepackage{amsthm}
\usepackage[cp1250]{inputenc}
\usepackage{setspace}
\usepackage[T1]{fontenc}
\usepackage{geometry}
\usepackage{fancyhdr}
\usepackage{graphicx}
\usepackage{color}
\usepackage{enumerate}
\usepackage{color}

\newcommand{\al}{\alpha}
\newcommand{\dd}{ {d-1} }

\newcommand{\lm}{\left|}
\newcommand{\J}{\mathcal{J}}
\newcommand{\rmo}{\right|}
\newcommand{\robnorm}[2]{\left| \! \left| \! \left| #1 \right| \! \right| \! \right|_{X,p,#2}}
\newcommand{\robnormp}[3]{\left| \! \left| \! \left| #1 \right| \! \right| \! \right|_{X,#3,#2}}

\newcommand{\N}{{\ensuremath{\mathbb N}}}

\newcommand{\Sd}{{\ensuremath{\mathbb S}}}
\newcommand{\E}{{\ensuremath{\mathbb E}}}
\newcommand{\Pro}{{\ensuremath{\mathbb P}}}

\newcommand{\R}{{\ensuremath{\mathbb R}}}

\newcommand{\lv}{\left\lVert}
\newcommand{\rv}{\right\rVert}

\newcommand{\ai}{a_{i_1,...,i_d}}
\newcommand{\aaa}{a_{i,j}}
\newcommand{\1}{\textbf{1}}

\newcommand{\il}{X^1_{i_1} \cdot \ldots \cdot X^d_{i_d}}
\newcommand{\Xii}{\prod_{k=1}^d X^k_i}
\newcommand{\Chaos} {\lv \sum_{i,j} \aaa X_i Y_j \rv_p}
\newcommand{\nx}{\hat{N}^X_i}
\newcommand{\ny}{\hat{N}^Y_j}
\newcommand{\m}{\hat{M}^X_i}
\newcommand{\kx}{\sum_{i} \nx(x_i) \leq p}
\newcommand{\ky}{\sum_{j} \ny(y_j) \leq p}
\newcommand{\wy}{\mathcal{E}}

\newcommand{\momp}[1]{\lv #1 \rv_p}
\newcommand{\norxy}[1]{\lv #1 \rv_{X,Y,p}}
\newcommand{\norxx}[1]{\lv #1 \rv_{X,X,p}}

\newcommand{\norx}[1] {\lv #1 \rv_{X,p}}
\newcommand{\norxp}[2] {\lv #1 \rv_{X,#2}}
\newcommand{\nory}[1]{\lv #1 \rv_{Y,p}}

\newcommand{\p}{{\lfloor p \rfloor }}

\newcommand{\id}{{i_1,\ldots,i_d}}

\newcommand{\eps}{\varepsilon}
\newcommand{\summ}[2]{\sum_{\stackrel{#1}{#2}}}
\newcommand{\mx}{\hat{M}^X_i}

\providecommand{\keywords}[1]{\textbf{\textit{Keywords: }} #1}

\providecommand{\klas}[1]{\textbf{\textit{AMS MSC 2010: }} #1}

\newtheorem{twr}[subsection]{Theorem}
\newtheorem{lem}[subsection]{Lemma}
\newtheorem{prep}[subsection]{Proposition}

\newtheorem{ex}[subsection]{Example}
\newtheorem{fak}[subsection]{Fact}
\newtheorem{cor}[subsection]{Corollary}

\newtheorem{rem}[subsection]{Remark}

\begin{document}

\title{Tail and moment estimates for a class of random chaoses of order two.}
\author{Rafa{\l} Meller\thanks{Research supported by the National Science Centre, Poland grant 2015/18/A/ST1/00553.}}
\date{}
\maketitle

\begin{abstract}
We derive two-sided bounds for moments and tails of random quadratic forms (random chaoses of order $2$), generated by independent symmetric random variables such that $\lv X \rv_{2p} \leq \al \lv X \rv_p$ for any $p\geq 1$ and some $\al\geq 1$. Estimates are deterministic and exact up to some multiplicative constants which depend only on $\al$.

 \\
\klas{60E15}
\end{abstract}

\section{Introduction} \label{sekcja1}
A (homogeneous) polynomial chaos of order $d$ is a random variable defined as
$$
S=\sum_{\id} \ai X_{i_1}\cdot \ldots \cdot X_{i_d},
$$
where $X_1,\ldots,X_n$ are independent random variables and $(\ai)_{1\leq i_1,\ldots,i_d\leq n}$ is an $d$-indexed symmetric array of real numbers with $0$'s on the generalized diagonals ($\ai=0$ whenever $i_k=i_l$ for some $k\neq l$). Such random variables occurs in many places in modern probability, e.g in  approximations of multiple stochastic integrals, 
Fourier-Walsh expansions of functions on the discrete cube (when the underlying variables $X_i$'s are independent Rademachers), in subgraph counts in random graphs (in this case $X_i$'s are zero-one random variables) or in statistical physics.

Chaoses of order $1$ are just linear combinations of independent random variables, classical objects
of probability theory. There are numerous bounds for moments and tails of sums of independent r.v's, in particular Lata{\l{}}a \cite{Chaos1} derived two-sided bounds
for $L_p$-norms of $\sum_i a_i X_i$ under assumptions that either $a_iX_i$ are nonnegative or $X_i$ are symmetric. The case $d\geq 2$ is much less understood. In the nonnegative case Lata{\l}a and {\L}ochowski
\cite{Latloch} established two-sided bounds for $\|S\|_p=(\E|S|^p)^{1/p}$ in the case, where the underlying
variables have log-concave tails. This result was generalized by the author \cite{Mel} to nonnegative chaoses
based on random variables satisfying the following moment condition
\begin{equation}
\label{mc}
\lv X_i \rv_{2p} \leq \al \lv X_i \rv_p \textrm{ for all } p\geq 1. 
\end{equation}
In the symmetric case two-sided bounds for $\|S\|_p$ are known only in few particular cases: Gaussian chaoses of any order \cite{Latgaus}, chaoses of any order based on symmetric  random variables with 
log-convex tails \cite{7}, chaoses of order $d\leq 3$ based on symmetric  random variables with log-concave tails \cite{AdLat,GluKw,Some}.

The aim of this article is to derive two-sided bounds for moments and tails of random quadratic forms (chaoses of order two) $\sum_{i\neq j} \aaa X_i X_j $  under the assumption \eqref{mc}. Since any symmetric random variable $X$ with log-concave tails satisfies \eqref{mc} with $\al=2$, this generalizes the previous result of Lata\l{}a \cite{Some}. Moreover \eqref{mc} arises naturally in the paper of  Lata\l{}a  and Strzelecka \cite{LatSt} as a sufficient condition (and even necessary in the i.i.d case)  for comparison of weak and strong moments of the random variable $\sup_{t \in T \subset \R^n} \sum t_i X_i $.  Lastly it is shown in \cite{Mel} (see Remark \ref{kryt} below) that if $\ln \Pro(|X|\geq Ktx)\leq t^\beta \Pro(|X|\geq x)$ for any $t,x\geq 1$ and some constants $K,\beta$, then   \eqref{mc} holds with $\al=\al(K,\beta)$. Thus this condition can be verified in many examples by an easy computation.

In the main proof of Theorem \ref{1} below we use the same idea as in \cite{Mel}. We replace variables $X_i$ by products of independent variables
with log-concave tails. However, the situation is much more difficult to handle than in the nonnegative case, since in the symmetric log-concave case two-sided moment bounds are known only for chaoses of small order. Instead we first establish Gluskin-Kwapien-type bounds for moments of linear combinations, decouple quadratic forms, apply conditionally bounds for $d=1$ and get to the point of estimating the $L_p$-norms of suprema of linear combinations of $X_i$'s. Although formulas are similar as in Lata{\l}a's paper 
\cite{Some}, we cannot use 
his approach since our random variables do not satisfy nice dimension-free concentration inequalities.
Instead we use a recent result of Lata{\l}a and Strzelecka \cite{LatSt} and reduce the question to
finding a right bound on $L_1$-norm of suprema. To treat this we use some ideas from  \cite{Latloch}
and \cite{AdLat}.
%To bound suprema of corresponding processes we can however use some ideas from \cite{Latloch} and a recent result of Lata{\l}a and Strzelecka \cite{LatSt}.

\section{Notation and main results} \label{sekcja2}
If $\textbf{v}$ is a deterministic vector in $\R^N$ (we do not exclude $N=\infty$) then $\lv \textbf{v} \rv_r$, $r \in [1,\infty]$, is its $l^r$ norm. 
We denote  by $g_1,g_2,\ldots$ independent $\mathcal{N}(0,1)$ random variables and by $\eps_1,\eps_2,\ldots$ independent symmetric $\pm 1$ random variables (Bernoulli sequence). We write $[n]$ for $\left\{1,\ldots,n\right\}$. If $X$ is a r.v. then $\lv X \rv_p:=\left( \E |X|^p \right)^{1/p}$. 
We say that a r.v. $X$ belongs to the class $\Sd(d)$ if $X$ is symmetric, $\lv X \rv_2=1/e$  and for every $p\geq 1$ $\lv X \rv_{2p} \leq 2^d \lv X \rv_{p}$ (the constant $1/e$ is chosen for technical reasons).
For a sequence $(X_i)_{i\geq 1}$ we define the function $N^X_i(t)=-\ln \Pro \left(|X_i|\geq |t| \right) \in [0,\infty]$ and set
\begin{equation}
\nx(t):=\begin{cases}t^2 &\textrm{ for } |t|\leq 1, \\ N^X_i(t) &\textrm{ for } |t| > 1. \end{cases} \label{nx}
\end{equation}
Analogously we define $N^Y_j(t),\ny(t)$. The following three norms   will play crucial role in this paper:
\begin{align}
&\lv (\aaa) \rv_{X,Y,p}=\sup \left\{\sum_{i,j} \aaa x_i y_j \ \Big{|} \ \kx,\ \ky \right\} \label{normxy}, \\
&\lv (a_i) \rv_{X,p} = \sup \left\{ \sum_i a_i x_i \ \Big{|}\ \kx \right\} ,\ \ \lv (a_j) \rv_{Y,p} = \sup \left\{ \sum_j a_j y_j \ \Big{|} \ \ky \right\}, \label{normy}
\end{align}
(see Lemma \ref{spraw} for the proof, that they  are norms).

By $C,c$ we denote a universal constant which may differ at each occurrence. We also write $C(d),c(d)$ if the constants may depend on the parameter $d$. We write $a \sim b$ ($a \sim^{d} b $ resp.) if  $b/C \leq a \leq Cb$ ($b/C(d) \leq a \leq C(d) b$ resp.).
 
Our main result is the following theorem.
\begin{twr} \label{1}
Assume that $(X_i),(Y_j)$ are independent random variables from the $\Sd(d)$ class. Then for any finite
matrix $(a_{i,j})$ and any $p\geq 1$,
\begin{equation}
\Chaos \sim^d \lv (\aaa) \rv_{X,Y,p}+ \lv \left(\sqrt{\sum_i \aaa^2}\right)_j \rv_{Y,p} + \lv \left(\sqrt{\sum_j \aaa^2}\right)_i \rv_{X,p}.\nonumber
\end{equation}
\end{twr}

We postpone the proof of Theorem \ref{1} till the end of this article and now
present some corollaries. The first one shows that property \eqref{mc} is preserved by the variable $\sum_{i,j} a_{i,j} X_i Y_j$.
\begin{cor} \label{wzrost}
Under the assumptions of Theorem \ref{1} we have
\begin{equation}
\lv \sum_{i,j} a_{i,j} X_i Y_j \rv_{2p}\leq C(d) \lv \sum_{i,j} a_{i,j} X_i Y_j \rv_{p} \label{rosnie}
\end{equation}
\end{cor}
\begin{proof}
Using Lemma \ref{gwiazda}  below twice we get
\begin{align*}
&\lv  (a_{i,j}) \rv_{X,Y,2p}  \\
&= \sup \left\{ \lv \left(\sum_{j} a_{i,j}y_j\right)_i  \rv_{X,2p} \ \Big{|} \ \sum_j \ny(y_j)\leq 2p \right\}\leq C(d) \sup \left\{ \lv \left(\sum_{j} a_{i,j}y_j\right)_i  \rv_{X,p} \ \Big{|} \ \sum_j \ny(y_j)\leq 2p \right\}  \\
&=C(d) \sup \left\{ \lv \left(\sum_{i} a_{i,j}x_i\right)_j  \rv_{Y,2p} \ \Big{|} \ \sum_i \nx(x_i)\leq p \right\}\leq C(d) \norxy{(a_{i,j})}.
\end{align*}
The above estimate together with Lemma \ref{gwiazda} and Theorem \ref{1} yields the assertion.
\end{proof}

Standard arguments show how to get from moment to tail bounds. 

\begin{cor} \label{taildec}
Under the assumptions of Theorem \ref{1} we have
\begin{equation}
\label{tail1dec}
\Pro\left(\left|\sum_{i,j} \aaa X_i Y_j\right|\geq C(d) \left( \lv (\aaa) \rv_{X,Y,p}+\nory{\left(\sqrt{\sum_i \aaa^2} \right)_j}+\norx{\left(\sqrt{\sum_j \aaa^2} \right)_i}\right) \right)
\leq e^{-p} 
\end{equation}
and 
\begin{equation}
\label{tail2dec}
\Pro\left(\left|\sum_{i,j} \aaa X_i Y_j\right|\geq c(d) \left( \lv (\aaa) \rv_{X,Y,p}+\nory{\left(\sqrt{\sum_i \aaa^2} \right)_j}+\norx{\left(\sqrt{\sum_j \aaa^2} \right)_i}\right) \right)\geq e^{-c(d)p}. 
\end{equation}
\end{cor}

\begin{proof}
The upper bound \eqref{tail1dec} is an immediate consequence of Chebyshev's inequality and 
Theorem \ref{1}. To establish the lower bound we have
\begin{align*}
&\Pro\left(\left|\sum_{i,j} \aaa X_i Y_j\right| \geq c(d) \left( \lv (\aaa) \rv_{X,Y,p}+\nory{\left(\sqrt{\sum_i \aaa^2} \right)_j}+\norx{\left(\sqrt{\sum_j \aaa^2} \right)_i}\right) \right)
\\
&\ \ \ \geq \Pro\left(\left|\sum_{i,j} \aaa X_i Y_j \right| \geq  \frac{1}{2} \lv \sum_{i,j} \aaa X_i Y_j \rv_p \right)\geq \left( 1-\frac{1}{2^p} \right)^2 \left(\frac{\lv \sum_{i,j} \aaa X_i Y_j \rv_p}{\lv \sum_{i,j} \aaa X_i Y_j \rv_{2p}} \right)^{2p}  \geq e^{-c(d)p},
\end{align*}
where the first inequality follows by Theorem \ref{1}, the second by the Paley-Zygmund inequality and the last one by \eqref{rosnie}.
\end{proof}

We formulate undecoupled versions of Theorem \ref{1} and Corollary \ref{taildec}.

\begin{cor} \label{2}
Let $X_1,X_2,\ldots$ be independent r.v's from the $\Sd(d)$ class and $(a_{i,j})$ be a finite matrix such that $a_{i,i}=0$ and $a_{i,j}=a_{j,i}$ for all $i,j$. Then for each $p\geq 1$,
\begin{equation}
\label{momundec}
\lv \sum_{i,j} \aaa X_i X_j \rv_p \sim^{d} \norxx{(\aaa)}+\norx{\left(\sqrt{\sum_i \aaa^2} \right)_j}, 
\end{equation}
\begin{equation}
\label{tail1undec}
\Pro\left(\left|\sum_{i,j} \aaa X_i X_j\right|\geq C(d) \left( \norxx{(\aaa)}+\norx{\left(\sqrt{\sum_i \aaa^2} \right)_j}\right) \right)\leq e^{-p} 
\end{equation}
and 
\begin{equation}
\label{tail2undec}
\Pro\left(\left|\sum_{i,j} \aaa X_i X_j\right|\geq c(d) \left( \norxx{(\aaa)}+\norx{\left(\sqrt{\sum_i \aaa^2} \right)_j}\right) \right)\geq e^{-c(d)p}.
\end{equation}
\end{cor}

\begin{proof}
Moment estimate \eqref{momundec} is an immediate consequence of Theorem \ref{1} and the Kwapie{\'n}
decoupling inequalities (Theorem \ref{A.8}).

%
%
%
%% Poni?szy remark do poprawy!!1!!!!
%
%

We may derive tail bounds from the moment estimates in the similar way as in the undecoupled case. Alternatively we may use the more general decoupling result of de la Pe{\~n}a  and Montgomery-Smith (Theorem \ref{A.9})  and get \eqref{tail1undec} and \eqref{tail2undec} from \eqref{tail1dec} and \eqref{tail2dec}.
\end{proof}

\begin{rem}
A simple approximation argument shows that  Theorem \ref{1} and Corollaries \ref{taildec}, \ref{2} hold for infinite square summable matrices $(a_{i,j})$.
\end{rem}

 We derive some examples from Corollary \ref{2}. Firstly we recover the special case of the Kolesko and Lata\l{}a result \cite[Example 3]{7}.

\begin{ex}\label{ex1}
Let $X_1,X_2,\ldots$ be independent r.v's with symmetric Weibull distribution with scale parameter 1 and shape
parameter $r \in (0, 1]$, i.e. for any $i$ $\Pro(|X_i| \geq t) = \exp(-t^r)$ for $t \geq 0$. Then for any $p\geq 1$ and any square summable matrix $(a_{i,j})$ such that $a_{i,i}=0$ and $a_{i,j}=a_{j,i}$ for all $i,j$ we have
\begin{align}
\lv \sum_{i,j} \aaa X_i X_j \rv_p &\sim^r p^{2/r} \sup_{i,j} |\aaa|+p^{1/r+1/2}\sup_i \sqrt{ \sum_j \aaa^2}+p\sup \left\{ \sqrt{ \sum_i \left(\sum_j \aaa x_j \right)} \ \Big{|} \ \lv x \rv_2 =1 \right\} \nonumber \\
&+\sqrt{p} \sqrt{\sum_{i,j} \aaa^2}. \label{dlugi}
\end{align}
\end{ex}
\begin{proof}
By direct computation one may check that $\lv X_i \rv_{2p} \leq 2^{1/r} \lv X_i \rv_p$ (it may be also checked that \eqref{mc} holds using Remark \ref{kryt}).
First observe that 
\begin{equation}
\norx{(v_i)}\sim \sqrt{p} \lv v \rv_2 + p^{1/r}\sup_i |v_i|. \label{loc11}
\end{equation}
Indeed we have that  $|x|^r\leq x^2$ for $|x|\geq 1$ and $|x|^r>x^2$ for $|x|<1$ so
$$
\norx{(v_i)}\sim \sqrt{p} \lv v \rv_2 + p^{1/r}\sup \left\{\sum_i v_i x_i \ \Big{|} \ \sum_i |x_i|^r \leq 1 \right\}.
$$
Obviously $\sup \left\{\sum_i v_i x_i \ \Big{|} \ \sum |x_i|^r \leq 1 \right\}\geq \sup_i |v_i|$. Since $r\in (0,1]$ we have $\sum_i |x_i|^r \geq ( \sum_i |x_i|)^{1/r}$ and as a result $\sup \left\{\sum_i v_i x_i \ \Big{|} \ \sum |x_i|^r \leq 1 \right\}= \sup_i |v_i|$ and \eqref{loc11} holds.

Iterating \eqref{loc11} we get
\begin{equation}
\norxx{(\aaa)} \sim  p^{2/r} \sup_{i,j} |\aaa|+p^{1/r+1/2}\sup_i \sqrt{ \sum_j \aaa^2}+p\sup \left\{ \sqrt{ \sum_i \left(\sum_j \aaa x_j \right)} \ \Big{|} \ \lv x \rv_2 =1 \right\}. \label{loc111}
\end{equation}

 The inequality \eqref{dlugi} follows by Corollary \ref{2}, \eqref{loc11} and \eqref{loc111}.
\end{proof}

Next example presents a situation when tails of $X_i$ are neither log-concave nor log-convex, so it cannot be deduced from previous results.

\begin{ex}\label{ex2}
Let $X_1,X_2,\ldots$ be i.i.d r.v's with distribution equal to $W\1_{|W|\leq R}$, where $R>1$ and $W$ be a symmetric Weibull distribution with scale parameter 1 and shape
parameter $r \in (0, 1]$. Assume that $(a_{i,j})_{i,j\geq 0}$ is a square summable matrix such that $a_{i,i}=0,\ a_{i,j}=a_{j,i}$. Denote  $A_i=\sqrt{\sum_j a_{i,j}^2}$ and
\begin{align*}
||| (a_{i,j}) |||_{R,r,p}&= \sup \left\{\sum_{i,j} \aaa x_i y_j \ \Big{|} \ \lv x \rv_2^2 \leq p,\ \lv x \rv_\infty \leq R,\ \lv y \rv_2^2 \leq p,\ \lv y \rv_\infty \leq R \right\}\\
&+ \sup \left\{\sum_{i,j} \aaa x_i y_j \ \Big{|} \ \lv x \rv_r^r \leq p,\ \lv x \rv_\infty \leq R,\ \lv y \rv_2^2 \leq p,\ \lv y \rv_\infty \leq R \right\} \\
&+ \sup \left\{\sum_{i,j} \aaa x_i y_j \ \Big{|} \ \lv x \rv_r^r \leq p,\ \lv x \rv_\infty \leq R,\ \lv y \rv_r^r \leq p,\ \lv y \rv_\infty \leq R \right\}.
\end{align*}
 Then
\begin{align*}
\lv \sum_{i,j} a_{i,j} X_i X_j \rv_p \sim \begin{cases}||| (a_{i,j}) |||_{R,r,p}+\sqrt{p}\sqrt{\sum_{i,j} a^2_{i,j}}+p^{\frac{1}{r}}A^*_1 &\textrm{for } 1\leq p \leq R^r \\
 ||| (a_{i,j}) |||_{R,r,p}+\sqrt{p}\sqrt{\sum_{i,j} a^2_{i,j}}+R \sum_{i \leq \frac{p}{R^r}}A^*_i &\textrm{for } R^r<p\leq R^2 \\
 ||| (a_{i,j}) |||_{R,r,p}+\sqrt{p} \sqrt{\sum_{i\geq \frac{p}{R^2}}(A^*_i)^2}+R\sum_{i\leq \frac{p}{R^r}} A^*_i &\textrm{for } R^2<p, \end{cases}
\end{align*}
where $(A^*_i)$ is a nonincreasing rearrangement of $(A_i)$.
\end{ex}

\begin{proof}
The assumptions of Corollary \ref{2} are satisfied since $\lv X_1 \rv_2 \sim^{r,R} 1$ and $X_1,X_2,\ldots$ satisfies \eqref{mc} with $\al=\al(r,R)$ (see Remark \ref{kryt}).
Iteration of the inequality \eqref{por} gives $||| (a_{i,j}) |||_{R,r,p}\sim \norxx{(\aaa)}$.
Corollary \ref{2} and Lemma \ref{osz} implies the assertion.

\end{proof}

\begin{rem}
In the Gaussian  and Rademacher case  Corollary \ref{2}  implies (see also Examples 1 and 2 in \cite{Some})
\begin{align}
\lv \sum_{i,j} \aaa g_i g_j \rv_p &\sim p \lv (\aaa) \rv_{l^2 \rightarrow l^2}+\sqrt{p} \lv (\aaa) \rv_2, \label{a1}\\
\lv \sum_{i,j} \aaa \eps_i \eps_j \rv_p &\sim \sup \left\{\sum_{i,j} \aaa x_i y_j \ \Big{|} \ \lv x \rv^2_2 \leq p, \lv y \rv^2_2\leq p,\lv x \rv_\infty \leq 1,\lv y \rv_\infty \leq 1 \right\}\label{a2}\\
&\ +\sum_{i\leq p}A^*_i+\sqrt{p}\sqrt{\sum_{i>p} (A^*_i)^2},
\end{align}
where $A^*_i$ is nonincreasing rearrangement of $A_i=\sqrt{\sum_j \aaa^2}$.
Neither \eqref{a1} nor \eqref{a2} can be expressed by a closed formula which do not involves suprema. Thus, there is no hope for any closed formulas in Examples \ref{ex1} and \ref{ex2}.
\end{rem}

The paper is organized as follows. In the next section we present some technical facts used in the main proof. In particular we specify what does it mean to "replace variables $X_i$ by products of independent variables
with log-concave tails". In Section \ref{sekcja3} we establish Gluskin-Kwapie\'n-type bounds for moments of linear combinations of $X_i$'s. 
In Section \ref{sekcja4} we obtain bounds for expected values 
of suprema and conclude the proof of Theorem \ref{1} in Section \ref{sekcja5}.
Unfortunately the proof of Theorem \ref{1} is very technical and depend on several technical results from many previous works. For the convenience of the reader we gather them in the Appendix, making our exposition self-contained.

%We also show how to apply result of \cite{LatSt} and compare $L_p$
%and $L_1$-norms of suprema of such combinations. 

\section{Preliminary facts}\label{prel}
We start with the crucial technical result from \cite{Mel}.

\begin{lem}\label{5}
If $X$ is from the $\Sd(d)$ class then there exists  symmetric i.i.d r.v's $X^1,\ldots,X^d$ on the extended probability space and a constant $t_0(d) \geq 1$ with the following properties:
\begin {align} %ni?ej by?o \label{eq3}
&C(d)(|X|+1) \geq |X^1\cdot \ldots \cdot X^d| \  \textrm{ and } \ C(d)(|X^1\cdot \ldots \cdot X^d|+1) \geq |X| \label{eq2}, \\
&X^1,\ldots,X^d \textrm{ have log-concave tails} \label{eq4}, \\
&M(t) \leq N(t^d) \leq M(C(d)t) \textrm{ for } t\geq t_0(d), \textrm{ where } M(t)=-\ln \Pro(|X^1| \geq t) \label{eq5}, \\
&\frac{1}{C(d)} \leq \E |X^1| \leq C(d) \label{eq6}, \\
&\inf\left\{t>0 \ | \ M(t)\geq 1 \right\}=1. \label{jed}
\end{align}
\end{lem}
\begin{proof}
From Lemma $3.3$ in \cite{Mel} we know that there exists symmetric i.i.d r.v's $X^1,\ldots,X^d$ which satisfy \eqref{eq2}-\eqref{eq6} and $M(t)=0$ for $t<t_0(d)$ where $t_0(d)>0$ ( see formula $(8)$ in \cite{Mel}). So  $\inf\left\{t>0 \ | \ M(t)\geq 1 \right\} \geq t_0(d)$. By Chebyshev's inequality $M(3\E |X^1|)\geq \ln(3)>1$. Combining it with \eqref{eq6} yields
$$\inf\left\{t>0 \ | \ M(t)\geq 1 \right\} \leq 3 \E |X^1| \leq C(d). $$
So we have proved that
$$0<c(d)\leq \inf\left\{t>0 \ | \ M(t)\geq 1 \right\} \leq C(d)<\infty.$$
The variables $X_i/\inf\left\{t>0 \ | \ M(t)\geq 1 \right\}$ satisfy \eqref{eq2}-\eqref{jed}.
\end{proof}

 Till the end of the paper  we assign to every $X_i$ from the $\Sd(d)$ class the r.v's $X^1_i,\ldots,X^d_i$ obtained by Lemma \ref{5}.

 Denote for $t\in \R$, $M^X_i(t)=-\ln \Pro \left(|X^1_i|\geq |t| \right) \in [0,\infty]$ and
\begin{equation}
\m(t)=\begin{cases}t^2 & \textrm{ for } |t|<1, \\ M^X_i(t) &\textrm{ for } |t|\geq 1.  \end{cases} \label{m}
\end{equation}

Observe that convexity of $M^X_i$ and the normalization condition \eqref{jed} imply
\begin{align}
\label{ulubionaetykieta1}
\hat{M}_i^X(\frac{t}{u})\leq \frac{\hat{M}_i^X(t)}{u} \mbox{ for } u\geq 1,
\end{align}
and 
\begin{align}
\label{ulubionaetykieta}
\hat{M}_i^X(t)=M^X_i(t)\geq |t| \mbox{ for } |t|\geq 1.
\end{align}

We define the following technical norms (the proof that they are norms is the same as for $\norxy{\ \cdot \ }$, see Lemma \ref{spraw})
\begin{align}
\robnorm{(a_i)}{1}&=\sup \left\{\sum_{i} a_i  x_i  \ \Big{|} \  \sum_{i} \m(x_i)\leq p \right\},  \\
\robnorm{(a_i)}{d}&=\sup \left\{\sum_{i} a_i  x^1_i \prod_{k=2}^d(1+x^k_i) \ \Big{|} \ \forall_{k=1,\ldots,d} \sum_{i} \m(x^k_i)\leq p \right\} \textrm{ for } d>1. \label{robocza}
\end{align}

\begin{lem} \label{porsum}
For any $p\geq 1$ we have
$$ \norx{(a_i)}\sim^d \robnorm{(a_i)}{d}.$$
\end{lem}

\begin{proof}
Without loss of generality we can  assume that $a_i$ are nonnegative. Let $t_0(d)$ be a constant from Lemma \ref{5}. We have
\begin{align*}
\norx{(a_i)} &\leq \sup \left\{\sum_i a_i b_i \1_{\{0\leq b_i<1\}} \ \Big{|} \ \sum_i \nx(b_i) \leq p \right\}+ \sup \left\{\sum_i a_i b_i \1_{\{t_0(d)^d>b_i\geq 1\}} \ |\ \sum_i \nx(b_i) \leq p \right\} \\
&\ \ \ +\sup \left\{\sum_i a_i b_i \1_{\{b_i\geq t_0(d)^d\}} \ |\ \sum_i \nx(b_i) \leq p \right\}=:I+I\!I+I\!I\!I.
\end{align*}

The equality $\nx(t)=\mx(t)$ for $|t|\leq 1$ implies

$$I\leq \sup \left\{\sum_i a_i x^1_i \ \Big{|} \  \sum_i \m(x^1_i) \leq p \right\}\leq \robnorm{(a_i)}{d}.$$
 
Since $\lv X_i \rv_2=1/e$, Chebyshev's inequality yields $\nx(s)\geq 1$ for $s\geq 1$. Hence
\begin{align*}
I\!I&= \sup \left\{ \sum_{i \in I} a_i b_i \1_{\{t_0(d)^d>b_{i}\geq 1\}} \ \Big{|} \ \sum_{i \in I} \nx(b_i) \leq p,\ |I|\leq \lfloor p \rfloor \right\} \leq t_0(d)^d \sup_{|I|=\lfloor p \rfloor}\sum_{i \in I} a_i \leq t_0(d)^d \robnorm{(a_i)}{d}. 
\end{align*}

To see the last inequality it is enough to take in \eqref{robocza} $x^1_i=\1_{\{i \in I \}}$ and $x^2_i=\ldots=x^d_i=0$.

From \eqref{eq5}  we obtain
\begin{align*}
I\!I\!I&\leq  \sup \left\{\sum_i a_i \left(b_i\right)^d \ | \ \sum_i \m(b_i) \leq p \right\} \leq \robnorm{(a_i)}{d},
\end{align*}
where to get the last inequality we take $x^1_i=\ldots=x^d_i=b_i$.

It remains to show
\begin{equation}
 \robnorm{(a_i)}{d} \leq C(d) \norx{(a_i)}. \label{l.1}
\end{equation}

By an easy computation
\begin{align}
\robnorm{(a_i)}{d} &\leq (1+C(d)t_0(d))^{d-1} \sup \left\{\sum_i a_i x^1_i \prod_{k=2}^d (1+x^k_i \1_{\{x^k_i>C(d)t_0(d)\}})\ \Big{|} \ \forall_{k} \sum_i \m(x^k_i) \leq p \right\} \nonumber \\
&\leq C(d) \sup \left\{\sum_i a_i x^1_i \1_{\{0 \leq x^1_i\leq C(d)t_0(d)\}}\ \Big{|} \ \sum_i \m(x^1_i) \leq p \right\}  \nonumber \\
&\ \ \ + C(d)\summ{I\subset [d]}{I \neq \emptyset} \sup \left\{\sum_i a_i \prod_{k\in I} x^k_i \1_{\{x^k_i>C(d)t_0(d)\}}\ \Big{|} \ \forall_{k} \sum_i \m(x^k_i) \leq p \right\} . \label{l.2}
\end{align}

Putting $y_i=x^1_i / (C(d) t_0(d))$ we see that
\begin{align}
&\sup \left\{\sum_i a_i x^1_i \1_{\{0 \leq x^1_i\leq C(d) t_0(d)\}}\ \Big{|} \ \sum_i \m(x^1_i) \leq p \right\} \nonumber \\
&= C(d) t_0(d) \sup \left\{\sum_i a_i y_i \1_{\{0 \leq y_i\leq 1\}}\ \Big{|} \ \sum_i \m(C(d)t_0(d)y_i) \leq p \right\}  \nonumber \\
&\leq C(d)\sup \left\{\sum_i a_i y_i \1_{\{0 \leq y_i\leq 1\}}\ \Big{|} \ \sum_i \nx(y_i) \leq p \right\} \leq C(d) \norx{(a_i)}, %\label{l.3}
\end{align}
where the first inequality follows by the monotonicity of the functions $(\mx)_i$.

%Obviously
%\begin{align}
%&\sup \left\{\sum_i a_i x^1_i \1_{\{x^1_i\leq C(d) t_0(d)\}}\ \Big{|} \ \sum_i \m(x^1_i) \leq p \right\} \nonumber \\
%&\leq \sup \left\{\sum_i a_i x^1_i \1_{\{x^1_i\leq 1\}}\ \Big{|} \ \sum_i \m(x^1_i) \leq p \right\} 
%+  \sup \left\{\sum_i a_i x^1_i \1_{\{1 \leq x^1_i\leq C(d) t_0(d)\}}\  \Big{|}  \ \sum_i \m(x^1_i) \leq p \right\}. \label{l.3}
%\end{align}
%We bound both expressions \eqref{l.3} by $\norx{a_i}$, similarly as in the first part of the proof (recall \eqref{nx} and \eqref{jed}).

 Now we estimate the second term in \eqref{l.2}. For a $I\subset [d],\ I\neq \emptyset$,
\begin{align}
& \sup \left\{ \sum_i a_i \prod_{k\in I} x^k_i \1_{\{x^k_i>C(d) t_0(d)\}}\ \Big{|} \ \forall_{k} \sum_i \m(x^k_i) \leq p \right\} \nonumber \\
&\leq \sup \left\{ \sum_i a_i \frac{1}{|I|} \sum_{k\in I} (x^k_i)^{|I|} \1_{\{x^k_i>C(d) t_0(d)\}}\ \Big{|} \ \forall_{k} \sum_i \m(x^k_i) \leq p \right\} \nonumber \\
&\leq  \sup \left\{ \sum_i a_i (x^1_i)^{d} \1_{\{x^1_i>C(d) t_0(d)\}}\ \Big{|} \  \sum_i \m(x^1_i) \leq p \right\} \nonumber \\
& \leq C(d) \sup \left\{ \sum_i a_i (x^1_i)^{d} \1_{\{ x^1_i> t_0(d)\}}\ \Big{|} \  \sum_i \m(C(d)x^1_i) \leq p \right\} \nonumber \\
&\leq C(d)  \sup \left\{ \sum_i a_i x^1_i \1_{\{x^1_i>(t_0(d))^d\}}\ \Big{|} \  \sum_i \nx(x^1_i) \leq p \right\}\leq C(d) \norx{a_i}, \label{l.4}
\end{align}
where to get the fourth inequality we used \eqref{eq5}. Estimates \eqref{l.2}-\eqref{l.4} imply \eqref{l.1}.
%To prove the last inequality in \eqref{l.4} we do the same as in \eqref{I}.
%Formulas \eqref{l.2}-\eqref{l.4} imply \eqref{l.1}.
\end{proof}

\begin{lem} \label{gwiazda}
There exists $C=C(d)$ such that for any $p,u\geq 1$ we have
$$\lv (a_i) \rv_{X,up} \leq C(d) u^d \lv (a_i) \rv_{X,p}.$$
\end{lem}
\begin{proof}
By Lemma \ref{porsum} it is enough to show the following inequality
\begin{align*}
\robnormp{(a_i)}{d}{up} \leq   u^d \robnormp{(a_i)}{d}{p}.
\end{align*}
The  inequality \eqref{ulubionaetykieta1} yields
\begin{align*}
\robnormp{(a_i)}{d}{up} &\leq  \sup\left\{ \sum_i a_i x^1_i \prod_{k=2}^d(1+x^k_i)\ \Big{|} \ \forall_k \sum_i \m(\frac{x^k_i}{u}) \leq p \right\} \\
&=  \sup \left\{\sum_i a_i (ux^1_i)\prod_{k=2}^d(1+ux^k_i)\ \Big{|} \ \forall_k \sum_i \m(x^k_i) \leq p  \right\} \leq  u^d \robnormp{(a_i)}{d}{p}.
\end{align*}
\end{proof}

%koniec sekcji wstepnej
%
%
%
%
%
%
%

\section{Moment estimates in the one dimensional case}\label{sekcja3}

In this section we will show two-sided bound for moments of linear combinations of r.v's from the $\Sd(d)$ class. 
\begin{twr}\label{jeden}
Let $X_1,X_2,\ldots$ be independent, symmetric random variables from the $\Sd(d)$ class. Then for any $p\geq 1$ and any finite sequence $(a_i)$ we have
\begin{equation}
\lv \sum_i a_i X_i \rv_p \sim^d \lv (a_i) \rv_{X,p}. \label{lepwzor}
\end{equation}

\end{twr}

Lata{\l{}}a \cite{Chaos1} (Theorem \ref{A.10}) derived bounds for moment of $\sum_i a_i X_i$ in a general case. However we were not able to deduce Theorem \ref{jeden} directly from it.
Instead below we present a direct tedious proof of \eqref{lepwzor}.

Since the r.v's $X_1,\ldots,X_n$ are symmetric and independent  without loss of the generality we may assume that $a_i \geq 0$.

\begin{lem}\label{6}
We have $\lv \sum_i a_i X_i \rv_p \sim^d \lv \sum_i a_i \prod_{k=1}^d X^k_i \rv_p$ for $p \geq 1$.
\end{lem}
\begin{proof}
 Let $(\eps_i)$ be a Bernoulli sequence, independent of $\left\{X_i,X^k_j\right\}_{i,j\geq 1, k\leq d}$ . Using the Jensen inequality and \eqref{eq6} we get
\begin{equation}
\lv \sum_i a_i \prod_{k=1}^d X^k_i \rv_p = \lv \sum_i a_i \eps_i \prod_{k=1}^d | X^k_i| \rv_p \geq  \lv \sum_i a_i \eps_i \E \prod_{k=1}^d |X^k_i| \rv_p \geq c(d)\lv \sum_i a_i \eps_i \rv_p. \nonumber %\label{loc}
\end{equation}
The contraction principle, Lemma \ref{5} and the triangle inequality yield
$$\lv \sum_i a_i X_i \rv_p=\lv \sum_i a_i \eps_i |X_i| \rv_p \leq \lv \sum_i a_i \eps_i C(d)\left( \left|\prod_{k=1}^d X^k_i\right|+1\right)\rv_p \leq C(d) \lv \sum_i a_i \prod_{k=1}^d X^k_i \rv_p .$$
The reverse bound may be established in an analogous way.
\end{proof}

Thus, to prove Theorem \ref{jeden} it is enough to properly bound $\lv \sum_i a_i \prod_{k=1}^d X^k_i \rv_p$ . The intuition behind Lemma \ref{6} is that we replaced each "big" r.v $X_i$ with a product of "smaller" pieces, which are easier to deal with.

Next lemma shows that Theorem 4.1 holds under the additional assumption that the support of the sum is small.

\begin{lem} \label{momkr}
For any $p\geq 1$ we have
$$\lv \sum_{i\leq p} a_i \prod_{k=1}^d X^k_i \rv_p \sim^d \robnorm{(a_i)_{i\leq p}}{d}. $$
\end{lem}
\begin{proof}
We will proceed by an induction on $d$. For $d=1$ lemma holds by the  Gluskin-Kwapie? bound \ref{gl}. Assume $d\geq 2$ and the assertion holds for $1,2,\ldots,d-1$. First we establish the lower bound for $\lv \sum_{i\leq p} a_i \prod_{k=1}^d X^k_i \rv_p$. We have
\begin{align*}
&C(d) \lv \sum_{i\leq p} a_i \prod_{k=1}^d X^k_i \rv_p  \geq \lv \sum_{i\leq p} a_i \prod_{k=1}^\dd X^k_i \E |X^d_i |\rv_p+ \lv \robnorm{(a_i X^d_i)_{i\leq p}}{d-1} \rv_p \\
&\geq \frac{1}{C(d)} \lv \sum_{i\leq p} a_i \prod_{k=1}^\dd X^k_i\rv_p
+ \sup \left\{ \lv \sum_{i \leq p } a_i x^1_i \prod_{k=2}^{d-1} (1+x^k_i) X_i   \rv_p \ \Big{|} \ \forall_{k\leq d-1} \sum_{i\leq p} \m(x^k_i)\leq p \right\}\\
&\geq \frac{1}{C(d)} \Bigg{(} \sup \left\{\sum_{i} a_i  x^1_i \prod_{k=2}^{d-1} (1+x^k_i) \ \Big{|} \ \forall_{k=1,\ldots,d-1} \sum_{i} \m(x^k_i)\leq p \right\}\\
 &\ \ \ +\sup \left\{\sum_{i} a_i  x^1_i \prod_{k=2}^{d-1}(1+x^k_i)x^d_i \ \Big{|} \ \forall_{k=1,\ldots,d} \sum_{i} \m(x^k_i)\leq p \right\} \Bigg{)} \\
&\geq \frac{1}{C(d)} \robnorm{(a_i)}{d},
\end{align*}
where the first inequality follows by Jensen's inequality and the induction assumption, the second by \eqref{eq6} and the third by  the induction assumption.

 Now we prove the upper bound. Using the right-continuity of $M^X_i$ we have
$$\Pro(M^X_i(|X^d_i|) \geq t)=\Pro\left(|X^d_i|\geq \left(M^X_i\right)^{-1}(t)\right)\leq e^{-t}\ \textrm{ for } t>0.$$
Therefore there exists nonnegative i.i.d r.v's $\wy_1,\ldots,\wy_p$ with the  density $ e^{-t}\1_{\{t>0\}}$ such that $M^X_i(|X^d_i|)\leq \wy_i$. Since $\sum_i^p \wy_i$ has the $\Gamma(p,1)$ distribution we obtain
\begin{align*}
\lv \sum_{i\leq p} M^X_i(|X^d_i|) \rv_q &\leq \lv \sum_{i \leq p} \wy_i \rv_q\leq Cq \ \textrm{ for } q\geq p.
\end{align*}
Since $M^X_i$ is convex, the above inequality implies for any $t\geq 1$,
\begin{align}
\Pro\left(\sum_{i\leq p } M^X_i \left(\frac{|X^d_i|}{Ct}\right) \geq p \right)  &\leq \Pro\left(\sum_{i\leq p } M^X_i \left(|X^d_i|\right) \geq Ctp \right) \leq (Ctp)^{-tp}\lv \sum_{i\leq p} M^X_i(|X^d_i|) \rv^{tp}_{tp} \nonumber \\
&\leq e^{-tp}. \label{prawd}
\end{align}

From \eqref{prawd}

\begin{align}
&\Pro\left( \robnorm{\left(a_i X^d_i\right)_{i \leq p}}{d-1} \geq Ct \robnorm{\left(a_i\right)_{i \leq p}}{d} \right) \nonumber \\
&=\Pro\left( \robnorm{\left(a_i \frac{X^d_i}{Ct}\right)_{i \leq p}}{d-1} \geq  \sup\left\{ \robnorm{\left(a_i(1+ x_i)\right)_{i \leq p}}{d-1}\ \Big{|} \ \sum_i \nx(x_i) \leq p  \right\}\right) \nonumber \\
&\leq \Pro\left(\sum_{i\leq p} \m \left(\frac{|X^d_i|}{Ct}\right) \1_{\{\frac{|X^d_i|}{Ct} > 1\}} \geq p \right) \leq
\Pro\left( \sum_{i\leq p} M^X_i \left(\frac{|X^d_i|}{Ct}\right)  \geq p \right) \leq e^{-tp}. \nonumber
\end{align}
Integration by parts  gives
\begin{align}
&\lv \robnorm{\left(a_i X^d_i\right)_{i \leq p}}{d-1}\rv_p \leq C\robnorm{\left(a_i\right)_{i \leq p}}{d}. \label{mo}
\end{align}
By the induction assumption and \eqref{mo},
\begin{align*}
\lv \sum_{i \leq p} a_i \prod_{k=1}^d X^k_i 
\rv_p \leq C(d) \lv \robnorm{\left(a_i X^d_i\right)_{i \leq p}}{d-1} \rv_p \leq C(d) \robnorm{\left(a_i\right)_{i \leq p}}{d},
\end{align*}
that concludes the proof of the induction step.

\end{proof}

\begin{rem}\label{zdolu}
Observe that in the proof of the lower bound in Lemma \ref{momkr} we  have not used the condition $i\leq p$.
\end{rem}

The idea of the following lemma is taken from \cite{Latloch}.

\begin{lem}\label{pozbiorach}
Let $p\geq 1$. Define
\begin{align}
T&=\left\{v \in \R^n \ \Big{|} \ \sum_i \mx(v_i) \leq p \right\}\cap \left\{ v \in \R^n \ | \ \forall_{i\leq n} \ |v_i|\geq 1 \textrm{ or } v_i=0 \right\}, \label{T} \\
U&=\bigcap_{l=1}^\infty \left\{v \in \R^n \  |  \  \mx(v_i)\leq l^3,\ i\in (2^{l}p,2^{l+1}p] \right\}\cap \left\{v \in \R^n \ |  \ \forall_{i \leq 2p} \ v_i=0  \right\} \cap T \label{U}, \\
V&=\left\{v \in \R^n \  |  \ \forall_{i\leq 2p} \ \mx(v_i)\geq 1 \textrm{ or } v_i=0 \right \} \nonumber \\
&\cap \bigcap_{l=1}^\infty \left\{v \in \R^n \  |  \  \mx(v_i)> l^3 \textrm{ or } v_i=0,\ i\in (2^{l}p,2^{l+1}p] \right\}  \cap T. \label{V}
\end{align}
If $(a_i)$ is a nonincreasing nonnegative sequence then
\begin{align}
&\E \sup_{x^1,\ldots,x^{d-1} \in U} \sum_i a_i \prod_{k=1}^{d-1} x^k_i |X^d_i| \leq C(d) \robnorm{(a_i)}{d}, \label{latU} \\
&\E \sup_{x^1,\ldots,x^{d-2} \in T,\ x^{d-1} \in V} \sum_i a_i \prod_{k=1}^{d-1} x^k_i |X^d_i| \leq C(d) \robnorm{(a_i)}{d}. \label{latV} 
\end{align}

\end{lem}
As we will see (in the next lemma) the main difficulty in proving Theorem \ref{jeden} is the proper estimation of 
$$\E \sup_{x^1,\ldots,x^{d-1} \in T} \sum_i a_i \prod_{k=1}^{d-1} x^k_i |X^d_i|. $$
 The key properties of sets $U,V$ are that $U,V \subset T \subset U+V$ and that we can prove \eqref{latU}, \eqref{latV} by some combinatorial arguments. The main difficulty in Lemma \ref{pozbiorach} is to figure how to decompose set $T$ (which was done in \cite{Latloch}).

\begin{proof}
We begin with \eqref{latU}. Using the fact that the sequence $(a_i)$ is nonnegative and \eqref{eq6} we obtain
\begin{align}
&\E \sup_{\stackrel{x^k \in U}{\ k=1,\ldots,d-1}} \sum_{i}  a_i \prod_{k=1}^{d-1}  x^k_i \lm X^d_i \rmo \leq \sum_{l=1}^\infty \E  \sup_{\stackrel{x^k \in U}{\ k=1,\ldots,d-1}}\sum_{i=2^{l}p+1}^{2^{l+1}p}  a_i \prod_{k=1}^{d-1} x^k_i \lm X^d_i \rmo \nonumber \\
&\leq C(d)  \sum_{l=1}^\infty l^{3(d-2)} \E \sup \left\{ \sum_{i=2^{l}p+1}^{2^{l+1}p}  a_i x_i \lm X^d_i \rmo \ \Big{|} \ \sum_i x_i \leq p,\ 1\leq x_i \leq l^3 \textrm{ or } x_i=0 \right\} \nonumber \\
&\leq C(d) \sum_{l=1}^\infty l^{3(d-2)} \E \sup \left\{ \sum_{i=2^{l}p+1}^{2^{l+1}p}  a_i x_i \lm X^d_i \rmo \ \Big{|} \ \sum_i x_i \leq l^3 \left\lceil \frac{p}{l^3} \right\rceil,\  x_i \in \{0, l^3\} \right\}. \label{kawal} 
\end{align}
Using Lemma \ref{A.6}
\begin{align}
&\E \sup \left\{ \sum_{i=2^{l}p+1}^{2^{l+1}p}  a_i x_i \lm X^d_i \rmo \ \Big{|} \ \sum_i x_i \leq l^3 \left\lceil \frac{p}{l^3} \right\rceil,\  x_i \in \{0, l^3\} \right\} \nonumber \\
&\leq \E \sup \left\{ \left( \sum_{i=2^{l}p+1}^{2^{l+1}p}  a_i x_i  \lm X^d_i\rmo -C\sum_{i=2^{l}p+1}^{2^{l+1}p}   a_i x_i   \right)_+ \ \Big{|} \ \sum_i x_i \leq l^3 \left\lceil \frac{p}{l^3} \right\rceil,\  x_i \in \{0, l^3\} \right\}+ Cl^3\left\lceil \frac{p}{l^3} \right\rceil  a_{2^{l}p+1} \nonumber \\ 
&\leq  C\left(p+\ln \left|\left\{v \in \R^{2^lp} \ \Big{|} \ \sum_i v_i \leq l^3 \left\lceil \frac{p}{l^3} \right\rceil,\  v_i \in \{0, l^3\} \right\} \right|\right)l^3a_{2^{l}p+1}+ C(l^3+p)a_{2^lp+1} \nonumber \\
&\leq C(p+ l^4)  a_{2^{l}p+1}. \label{oszkwa}
\end{align}
The last inequality follows by the simple estimate
$$\binom{2^lp}{\left\lceil\frac{p}{l^3}\right\rceil}\leq \left(\frac{2^l p e}{\left\lceil\frac{p}{l^3}\right\rceil} \right)^{\left\lceil\frac{p}{l^3}\right\rceil}\leq C^{p+l}.$$
Combining \eqref{kawal} and \eqref{oszkwa} we obtain 
\begin{align}
&\E \sup_{x^k \in U,\ k=1,\ldots,d-1} \sum_{i}  a_i \prod_{k=1}^{d-1} x^k_i \lm X^d_i \rmo \leq C(d)  \sum_{l=1}^\infty l^{3d} p a_{2^{l}p+1} \leq C(d)\sum_{l=1}^\infty l^{3d} p \frac{\sqrt{\sum_{i>p} a_i^2}}{\sqrt{2^{l-1}p}} 
\leq C(d)  \sqrt{p} \sqrt{\sum_{i>p} a^2_i}. \label{oszacowanieU} 
\end{align}
To finish the proof of \eqref{latU} it is enough to observe that by Lemma \ref{rad}
\begin{align*}
\frac{\sqrt{p}}{2} \sqrt{\sum_{i>p} a^2_i}&\leq  \sup \left\{\sum_i a_i t_i \ \Big{|} \ \sum_i t^2_i\leq p,\ \forall_i |t_i| \leq 1 \right\}\leq  \robnorm{(a_i)}{d}. 
\end{align*}

Now we show \eqref{latV}. Let 
\begin{equation}
\J=\left\{ I \subset \N \  |  \ |I|\leq p,\ \forall_{l \in \N}  \ |I \cap [2^{l}p+1,2^{l+1}p)| \leq \frac{p}{l^3} \right\} \label{trojkat}
\end{equation}
 (in particular $\J$ contains any subset of $[2p]$ of cardinality not greater than $p$). Applying the inequality $ \binom{n}{m} \leq \left( \frac{en}{m} \right)^m$, we get an estimate of the cardinality of $\J$
\begin{equation}
|\J|\leq 2^{2p} \prod_{l=1}^{\left\lfloor p^{\frac{1}{3}}  \right\rfloor} \left(\frac{C2^{l}p}{\left\lfloor p/l^3  \right\rfloor} \right)^{\left\lfloor p/l^3  \right\rfloor} \leq C^p \prod_{l=1}^\infty (2^{l} l^3)^{ p/l^3}\leq C^p. \label{mocJ}
\end{equation}

Take any $I\in \J$. We obtain (see Lemma \ref{momkr})
\begin{align}
&\E \robnorm{(a_i X^d_i)_{i 
\in I}}{d-1} \leq C(d) \robnorm{(a_i)_{i 
\in I}}{d}. \label{poi}
\end{align}
 Using the definition of $V$, \eqref{mocJ} and \eqref{poi}
\begin{align*}
&\E \sup_{} \left\{\sum_i  a_i    \prod_{k=1}^{d-1} x^k_i X^d_i 
\ \Big{|} \ \forall_{k\leq d-2} \sum_i \mx(x^k_i)\leq p,\ x^{d-1} \in V \right\}  \\
&\leq \left( \sum_{I\in \J} \left(\E \robnorm{(a_i X^d_i)_{i 
\in I}}{d-1}\right)^p \right)^{1/p}  \leq C(d) \sup_{I \in \J} \robnorm{(a_i)_{i 
\in I}}{d}\leq C(d) \robnorm{(a_i)}{d}. 
\end{align*}

\end{proof}

\begin{lem} \label{7.5}
For any $d \in \N$ the following holds
\begin{align}
 \E \robnorm{(a_i X^d_i)}{d-1} \leq C(d)  \robnorm{(a_i)}{d}. \label{5.3.1}
\end{align}
\end{lem}
\begin{proof}
Without loss of the generality we may assume that the sequence $(a_i)$ is nonincreasing and recall that in this section $a_i$ are nonnegative.
We proceed by an induction on $d$. If $d=2$ then Corollary \ref{A.7} (with $\hat{a}_{i,j}=a_i \1_{\{i=j\}}$) implies the assertion. 
Assume that \eqref{5.3.1} holds for any $2,3,\ldots,d-1$. Obviously,
\begin{align}
 \E \robnorm{(a_i X^d_i)}{d-1}&\leq \E \sup \left\{\sum_i a_i  x^1_i \1_{\{ |x^1_i| \leq 1 \}} \prod_{k=2}^{d-1} (1+x^k_i)  X^d_i  \ \Big{|} \ \forall_{k\leq d-1} \sum_i \mx(x^k_i)\leq p \right\} \nonumber \\
 &\ +\E \sup \left\{\sum_i a_i  x^1_i \1_{\{|x^1_i| > 1\}} \prod_{k=2}^{d-1}(1+ x^k_i) X^d_i \ \Big{|} \ \forall_{k\leq d-1} \sum_i \mx(x^k_i)\leq p \right\} \nonumber \\
&=: S_1+S_2.\label{rozbicie}
\end{align}

We have
\begin{align}
S_1&\leq \E \sup \left\{\sum_i  a_i  x^1_i    X^d_i \ \Big{|} \ \sum_i \mx(x^1_i)\leq p \right\} \nonumber \\
&\ \ \ +\summ{I \subset \left([d-1]\setminus \left\{1\right\}\right)}{I \neq \emptyset} \E \sup \left\{\sum_i  a_i   \prod_{k \in I}  x^k_i  X^d_i  \ \Big{|} \ \forall_{2 \leq k\leq d-1} \sum_i \mx(x^k_i)\leq p \right\} \nonumber \\
&\leq 2^{d-2} \E \sup \left\{\sum_i  a_i  x^2_i  \prod_{k=3}^{d-1} (1+ x^k_i ) X^d_i   \ \Big{|} \ \forall_{2 \leq k\leq d-1} \sum_i \mx(x^k_i)\leq p \right\} \leq C(d)\robnorm{(a_i)}{d}, \label{s1}
\end{align}
where the last inequality follows by the induction assumption.

Now we bound $S_2$. Since $a_i$ are nonnegative,
\begin{align}
S_2 &\leq 2^d \E \sup \left\{\sum_i a_i  x^1_i \1_{\{x^1_i > 1\}} \prod_{k=2}^{d-1}(1+ x^k_i\1_{\{x^k_i> 1\}}) \left|X^d_i\right| \ \Big{|} \ \forall_{k\leq d-1} \sum_i \mx(x^k_i)\leq p \right\} \nonumber \\
&\leq 2^d \E \sup \left\{\sum_i  a_i   \prod_{k=1}^{d-1}  x^k_i \1_{\{x^k_i> 1\}} \lm X^d_i \rmo \ \Big{|} \ \forall_{k\leq d-1} \sum_i \mx(x^k_i)\leq p \right\} \nonumber \\
&\ \ \ + 2^d \summ{I \subsetneq [d-1]}{I \neq \emptyset} \E \sup \left\{\sum_i  a_i   \prod_{k \in I}  x^k_i \1_{ \{ x^1_i  > 1 \}} \lm X^d_i \rmo \ \Big{|} \ \forall_{k\leq d-1} \sum_i \mx(x^k_i)\leq p \right\}. \label{brzydka}
\end{align} 
We can bound the second term in \eqref{brzydka} by using the induction assumption. So it is enough to show that 
\begin{align}
&\E \sup \left\{\sum_i  a_i   \prod_{k=1}^{d-1}  x^k_i \1_{\{x^k_i > 1\}} \lm X^d_i \rmo \ \Big{|} \ \forall_{k\leq d-1} \sum_i \mx(x^k_i)\leq p \right\} \leq C(d) \robnorm{(a_i)}{d}.
\label{zmiana}
\end{align}
Let $T,U,V$ be the sets defined in \eqref{T}-\eqref{V}. Since $T \subset U+V$ we have
\begin{align}
&\E \sup \left\{\sum_i  a_i   \prod_{k=1}^{d-1}  x^k_i \1_{\{x^k_i > 1\}}  \lm X^d_i \rmo \ \Big{|} \ \forall_{k\leq d-1} \sum_i \mx(x^k_i)\leq p \right\} \nonumber \\
&\leq \E \sup_{x^1,\ldots,x^{d-1} \in U} \sum_i a_i \prod_{k=1}^{d-1} x^k_i |X^d_i|+\sum_{\stackrel{I \subset [d-1]}{I\neq \emptyset}} \E \sup_{x^i \in U,\ i\in I}\sup_{x^i\in V,\ i \notin I} \sum_i a_i \prod_{k=1}^{d-1} x^k_i|X^d_i| \nonumber \\
&\leq  \E \sup_{x^1,\ldots,x^{d-1} \in U} \sum_i a_i \prod_{k=1}^{d-1} x^k_i |X^d_i|+\left(2^{d-1}-1\right)\E \sup_{x^1,\ldots,x^{d-2} \in T,\ x^{d-1} \in V} \sum_i a_i \prod_{k=1}^{d-1} x^k_i |X^d_i|. \label{cosik}
\end{align}
In the last inequality we used the symmetry and inclusions $U,V\subset T$. Now \eqref{zmiana} follows by \eqref{cosik}, \eqref{latU} and \eqref{latV}.
\end{proof}

\begin{lem}\label{7}
We have 
\begin{align}
\lv \sum_i a_i \Xii \rv_p \sim^d \robnorm{(a_i)}{d}. \label{7.1}
\end{align}

\end{lem} 
\begin{proof}
We prove \eqref{7.1} by an induction. In case of $d=1$ it follows by Gluskin-Kwapie{\'n} bound \ref{gl}. Assume \eqref{7.1} hold for any $1,2,\ldots,d-1$. The lower bound of $\lv \sum_i a_i \Xii \rv_p$ follows by Remark \ref{zdolu}.

Now we prove
$$ \lv \sum_i a_i \Xii \rv_p \leq C(d) \robnorm{(a_i)}{d}. $$

Induction assumptions imply
\begin{align}
&\lv \sum_i a_i \Xii \rv_p \leq C(d) \lv \robnorm{(a_i X^d_i)}{d-1} \rv_p \nonumber \\
&\leq C(d) \E \robnorm{(a_i X^d_i)}{d-1} 
+C(d)\sup  \left\{ \lv \sum_i a_i  x^1_i \prod_{k=2}^{d-1} (1+x^k_i) X^d_i\rv_p \ \Big{|} \ \forall_{k\leq d-1} \sum_i \mx(x^k_i)\leq p \right\} , \label{pier} 
\end{align}
where in the second inequality we used  Theorem \ref{9} (it is an easy exercise that if $Z$ is symmetric random variable with log-concave tail then $\lv Z \rv_{2p} \leq 2 \lv Z \rv_p$ for $p\geq 1$) . The
Gluskin-Kwapie{\'n} estimate (i.e the first step of the induction) gives
\begin{align}
\sup  \left\{ \lv \sum_i a_i  x^1_i \prod_{k=2}^{d-1} (1+x^k_i) X^d_i\rv_p \ \Big{|} \ \forall_{k\leq d-1} \sum_i \mx(x^k_i)\leq p \right\}\leq C(d) \robnorm{(a_i)}{d}. \label{drugczesc}
\end{align}

The assertion follows by \eqref{pier}, Lemma \ref{7.5} and \eqref{drugczesc}.
\end{proof}

Theorem \ref{jeden} follows by Lemmas \ref{6}, \ref{7} and \ref{porsum}.

Observe that Lemma \ref{gwiazda} and \eqref{lepwzor} imply that moments of $\sum_i a_i X_i$ cannot grow too quickly, namely
\begin{equation}
\lv \sum_i a_i X_i \rv_{2p} \leq C(d) \momp{\sum_i a_i X_i} \textrm{ for } p
\geq 1.\label{momsum} 
\end{equation}

\section{Estimates for suprema of processes}\label{sekcja4}
\newcommand{\gv}{\textbf{v}}
The main difficulty in proving Theorem \ref{1} is to properly bound $\E \sup_{x \in \Hat{T}} \sum_{i,j} \aaa x_i Y_j$
in two cases:  
\begin{enumerate}
\item $\Hat{T}=\left\{x \in \R^n\ \Big{|} \ \sum_i \mx(x_i)\leq p \right\} \subset \sqrt{p}B^n_2+pB^n_1$,
\item $\Hat{T}=\left\{ \left(\prod_{k=1}^d x^k_i \1_{\{x^k_i\geq 1\}}\right)_i \in \R^n\ \Big{|} \ \forall_{k\leq d} \sum_i \mx(x^k_i)\leq p \right\} $.
\end{enumerate}
We start with two lemmas  which are responsible for solving the second case. The rest of the section is devoted for developing some decomposition lemmas. They will be used to handle the first case.

Firstly, we show that Theorem \ref{1} holds under additional assumption that the support of the sum is small.
\begin{lem}\label{momkrt}
For any $p\geq 1$, and set $I\subset [n],\ |I|\leq p$,
\begin{align}
\lv \sum_{i \in I, j} \aaa X_i Y_j \rv_p \leq C(d) \norxy{(\aaa)_{i\in I,j}}.  \label{momkro}
\end{align}
\end{lem}
\begin{proof}
%By using Theorems \ref{jeden} and \ref{9} 
%\begin{align*}
%\lv \sum_{i \in I, j} \aaa X_i Y_j \rv_p &\leq C(d) \left( \E \nory{\left( \sum_{i \in I} \aaa X_i \right)_j}+ \sup \left\{\lv \sum_{i \in I, j} \aaa X_i y_j \rv_p \ \Big{|} \ \sum_i \ny(y_i)\leq p  \right\} \right) \\
%&\leq C(d) \left(\E \nory{\left( \sum_{i \in I} \aaa X_i \right)_j}  + \norxy{(\aaa)_{i\in I,j}} \right) \\
%&\leq C(d) \norxy{(\aaa)_{i\in I,j}}.
%\end{align*}
%In the last inequality we used Lemma \ref{prockrt} and \eqref{lepwzor}. So \eqref{momkro} holds.
We have
\begin{align*}
\lv \sum_{i \in I,j} \aaa X_i Y_j \rv_p &\leq C(d) \left( \E \sup_{\sum_j \ny(y_j)\leq p} \lm \sum_{i \in I,j} \aaa X_i y_j \rmo^p \right)^{\frac{1}{p}}\leq C(d)\sup_{\sum_j \ny(y_j)\leq p} \left( \E  \lm \sum_{i \in I,j} \aaa X_i y_j \rmo^p \right)^{\frac{1}{p}} \\
&\leq C(d) \norxy{(\aaa)_{i \in I,j}},
\end{align*}
where the first inequality follows by a conditional application of Theorem \ref{jeden}, the second one by Fact \ref{prockrt} and the last one by Theorem \ref{jeden}.

\end{proof}

From now till the end of this section we assume (without loss of generality) the following condition
\begin{align}
\textrm{the functions } i\rightarrow \sum_j \aaa^2,\ j \rightarrow \sum_i \aaa^2 \textrm{ are nonincreasing.} \label{monot}
\end{align}

\begin{lem}[cf. Lemma \ref{7.5}] 
Let $T,U,V$ be the sets defined in \eqref{T}-\eqref{V}. Then
\begin{align}
&S_1:=\E \sup_{x^1,\ldots,x^d \in U} \sum_{i,j} \aaa \prod_{k=1}^d x^k_i  Y_j \leq C(d) \nory{\left(\sqrt{\sum_i \aaa^2}\right)_j}, \label{poU} \\
&S_2:=\E \sup_{x^1,\ldots,x^{d-1}\in T,x^d\in V } \sum_{i,j} \aaa \prod_{k=1}^d x^k_i Y_j \leq C(d) \norxy{(\aaa)_{i,j}}. \label{poV}
\end{align}
\end{lem}
\begin{proof}
We begin with \eqref{poU}. By the symmetry of $\mx$,
\begin{align*}
S_1=\E \sup_{x^1,\ldots,x^d \in U} \sum_{i}\prod_{k=1}^d x^k_i \left| \sum_j \aaa Y_j \right| \leq \sum_{l=1}^\infty  \E \sup_{x^1,\ldots,x^d \in U} \sum_{2^lp<i\leq 2^{l+1}p}  \prod_{k=1}^d x^k_i \left| \sum_j \aaa Y_j \right|. 
\end{align*}
So by \eqref{ulubionaetykieta} 

\begin{align*}
S_1&\leq \sum_{l=1}^\infty  \E \sup \left\{ \sum_{2^lp<i\leq 2^{l+1}p}  \prod_{k=1}^d x^k_i \Big{|} \sum_j \aaa Y_j \Big{|} \ \Bigg{|} \ \forall_{1\leq k \leq d} \sum_{i=2^lp+1}^{2^{l+1}p} |x^k_i|\leq p,\ 1\leq |x_i| \leq l^3 \textrm{ or } x_i=0 \right\} \\
&\leq \sum_{l=1}^\infty l^{3d} \E \sup_{\stackrel{|I|=p}{I\subset (2^lp,2^{l+1}p]}} \sum_{i \in I} \left|\sum_j \aaa Y_j \right|\leq \sum_{l=1}^\infty l^{3d}p^{\frac{3}{4}}  \E \sup_{\stackrel{|I|=p}{I\subset (2^lp,2^{l+1}p]}} \left( \sum_{i \in I} \left(\sum_j \aaa Y_j\right)^4 \right)^{\frac{1}{4}}\\
&\leq \sum_{l=1}^\infty l^{3d}p^{\frac{3}{4}}  \E \left( \sum_{2^lp<i\leq 2^{l+1}p} \left(\sum_j \aaa Y_j\right)^4 \right)^{\frac{1}{4}},
\end{align*}
where the third inequality follows by the H{\"o}lder inequality. Now using the Jensen inequality and \eqref{momsum}  
\begin{align}
&\sum_{l=1}^\infty l^{3d}p^{\frac{3}{4}}  \E \left( \sum_{2^lp<i\leq 2^{l+1}p} \left(\sum_j \aaa Y_j\right)^4 \right)^{\frac{1}{4}} \leq \sum_{l=1}^\infty l^{3d}p^{\frac{3}{4}}  \left( \sum_{2^lp<i\leq 2^{l+1}p} \lv \sum_j \aaa Y_j \rv_4^4 \right)^{\frac{1}{4}} \nonumber \\
&\leq C(d) \sum_{l=1}^\infty l^{3d}p^{\frac{3}{4}}  \left( \sum_{2^lp<i\leq 2^{l+1}p} \lv \sum_j \aaa Y_j \rv_2^4 \right)^{\frac{1}{4}} \leq C(d) \sum_{l=1}^\infty l^{3d}p^{\frac{3}{4}}  \left( \sum_{2^lp<i\leq 2^{l+1}p} \left(\sum_j \aaa^2 \right)^2 \right)^{\frac{1}{4}}. \label{wzor}
\end{align}
Denote $B=\sqrt{\sum_{\stackrel{i\geq p}{j\geq 1}} \aaa^2}$. Since $i\rightarrow \sum_j \aaa^2$ is nonincreasing we have that
$$\sum_j \aaa^2 \leq \frac{B^2}{i-p} \textrm{ for } i>p.$$
Using the above estimate in \eqref{wzor} gives
\begin{align*}
\sum_{l=1}^\infty l^{3d}p^{\frac{3}{4}}  \left( \sum_{2^lp<i\leq 2^{l+1}p} \left(\sum_j \aaa^2 \right)^2 \right)^{\frac{1}{4}}&\leq \sum_{l=1}^\infty l^{3d}p^{\frac{3}{4}}  \left( \sum_{i>2^lp} \frac{B^4}{(i-p)^2} \right)^{\frac{1}{4}}\\
&\leq  \sum_{l=1}^\infty p^{\frac{3}{4}} l^{3d} \left( \frac{CB^4}{2^lp}\right)^\frac{1}{4}\leq C(d)\sqrt{p} B.
\end{align*}
By \eqref{monot} and Lemma \ref{rad} we have
$$\frac{\sqrt{p}}{2}B\leq \sup\left\{\sum_i \sqrt{\sum_j \aaa^2}t_i \ \Big{|} \ \sum_i t^2_i \leq p,\ |t_i|\leq 1 \right\}\leq  \nory{\left(\sqrt{\sum_i \aaa^2} \right)_j} .$$
As a consequence
\begin{align*}
S_1 \leq C(d)\nory{\left(\sqrt{\sum_i \aaa^2} \right)_j}.
\end{align*}

Now we bound $S_2$. Let $\J$ be defined by \eqref{trojkat} and let $I\in \J$ be arbitrary. Using conditionally \eqref{lepwzor}, Lemma \ref{porsum} and the Jensen inequality
\begin{align*}
&\E \sup_{x^1,\ldots,x^{d-1}\in T, x^d \in V} \sum_{i \in I} \prod_{k=1}^d x^k_i \left| \sum_j \aaa Y_j \right| \leq \E \sup_{x^1,\ldots,x^d \in T} \sum_{i \in I} \prod_{k=1}^d x^k_i \left| \sum_j \aaa Y_j \right| \\
&\leq C(d) \E_Y \left( \E_X \left| \sum_{i \in I,j} \aaa X_i Y_j \right|^p \right)^{\frac{1}{p}} \leq C(d)\lv \sum_{i \in I,j} \aaa X_i Y_j \rv_p.
\end{align*}
By Lemma \ref{momkrt} 
\begin{align*}
\lv \sum_{i \in I,j} \aaa X_i Y_j \rv_p \leq C(d) \norxy{(\aaa)_{i \in I,j}}. 
\end{align*}
So we conclude that
\begin{align}
\E \sup_{x^1,\ldots,x^{d-1}\in T, x^d \in V} \sum_{i \in I} \prod_{k=1}^d x^k_i \left| \sum_j \aaa Y_j \right| \leq C(d) \norxy{(\aaa)_{i \in I,j}} \leq C(d) \norxy{(\aaa)_{i,j}}. \label{cos}
\end{align}

By \eqref{mocJ} and \eqref{cos}  
\begin{align*}
S_2&\leq C(d) \left(\sum_{I \in \J} \left( \E \sup_{x^1,\ldots,x^{d-1}\in T,x^d\in V } \sum_{i\in I} \prod_{k=1}^d x^k_i \left| \sum_j \aaa Y_j \right| \right)^p \right)^\frac{1}{p}  \\
&\leq C(d) \sup_{I \in \J} \norxy{(\aaa)_{i \in I,j}} \leq C(d) \norxy{(\aaa)_{i,j}}. 
\end{align*}

\end{proof}

As it was announced earlier, we proceed with the study of decomposition lemmas.

It is well know (cf. Lemma 3 in \cite{Latgaus}) that if  $T=\bigcup_{k=1}^m T_l$ then
$$\E \sup_{t \in T} \sum_i t_i g_i \leq \max_{k \leq m} \E \sup_{t \in T_k} \sum_i t_i g_i +
C \sqrt{\log(m)} \sup_{s,t \in T} \lv \sum_i (t_i-s_i)g_i \rv_2. $$
We will generalize this formula to any variables from the $\Sd(d)$ class.

\begin{cor} \label{10}
For any $p \geq 1$ we have
$$
\Pro\left(\sup_{t \in T} \left| \sum_i t_i X_i \right| \geq C(d) \left( \E \sup_{t \in T} \left|\sum_i t_i X_i \right| + \sup_{t \in T}\norx{(t_i)}    \right) \right)\leq e^{-p}.$$
\end{cor}
\begin{proof}
It is a simple consequence of Theorem \ref{9}, \eqref{lepwzor} and Chebyshev's inequality.
\end{proof}

%To this end we need the following consequence of Lata?a and Strzelecka result and \eqref{momsum}.
%\begin{twr}[Theorem 1.1 in \cite{LatSt}] \label{9}
%For any non empty set $T\subset \R^n$ and $p \geq 1$ we have
%$$\momp{\sup_{t \in T} \left|\sum_i t_i X_i \right| } \leq C(d) \left( \lv \sup_{t \in T} \sum_i t_i X_i \rv_1 + \sup_{t \in T} \lv \sum_i t_i X_i \rv_p \right).$$
%\end{twr}

\begin{lem}\label{11}
Let $T=\bigcup_{k=1}^m T_k$, $m\geq 8$. Then 
$$\E \sup_{t \in T} \sum_i t_i X_i \leq C(d) \left(\max_{k \leq m} \E \sup_{t \in T_k} \sum_i t_i X_i + \sup_{s,t \in T }\norxp{(t_i-s_i)_i}{\ln(m)}  \right).$$

\end{lem}

\begin{proof}
We choose any $s\in T$. Since $\E X_i=0$ we have
$$\E \sup_{t \in T} \sum_i t_i X_i= \E \max_{k\leq m} \sup_{t \in T_k} \sum_i (t_i-s_i) X_i \leq \E \max_{k \leq m} \sup_{t \in T_k} \left|\sum_i (t_i-s_i)X_i \right|.$$
Corollary \ref{10} and the union bound yield for $u \geq 1$,
\begin{align*}
&\Pro \left(\max_{k\leq m} \sup_{t \in T_k} \left| \sum_i (t_i-s_i)X_i \right|\geq C(d) \left[ \max_{k} \E \sup_{t\in T_k} \left|\sum_i (t_i-s_i)X_i \right| +\norxp{(t_i-s_i)_i}{u\ln(m)} \right] \right) \\
&\leq m e^{-u \ln (m)}\leq 4^{1-u}.
\end{align*}

 Lemma \ref{gwiazda} implies $\norxp{(t_i-s_i)_i}{u\ln(m)} \leq C(d) u^d \norxp{(t_i-s_i)_i}{\ln(m)}$. Hence

$$\Pro \left( \sup_{t \in T} \left| \sum_i (t_i-s_i)X_i \right|\geq  C(d) \left[ \max_{k} \E \sup_{t\in T_k} \left|\sum_i (t_i-s_i)X_i \right| + u^d \sup_{s,t \in T} \norxp{(t_i-s_i)_i}{\ln(m)} \right] \right)\leq 4^{1-u}. $$
Integration by parts gives
$$\E \sup_{t \in T} \left|\sum_i (t_i-s_i)X_i \right| \leq C(d) \left(\max_{k} \E \sup_{t\in T_k} \left|\sum_i (t_i-s_i)X_i \right| + \sup_{s,t \in T}  \norxp{(t_i-s_i)_i}{\ln(m)} \right).$$
So to finish the proof it is enough to show that 
\begin{equation}
 \E \sup_{t\in T_k} \left|\sum_i (t_i-s_i)X_i \right| \leq C(d) \left(\E \sup_{t \in T_k} \sum_i t_i X_i+ \sup_{s,t \in T}\norxp{(t_i-s_i)_i}{\ln(m)}  \right). \label{polo2}
\end{equation}  
Let $z \in T_k$. We have
\begin{align}
&\E \sup_{t\in T_k} \left|\sum_i (t_i-s_i)X_i \right| \leq \E \sup_{t\in T_k} \left|\sum_i (t_i-z_i)X_i \right|+ \E  \left|\sum_i (z_i-s_i)X_i \right| \nonumber \\
&\leq \E \sup_{t\in T_k} \left|\sum_i (t_i-z_i)X_i \right|+\lv \sum_i (z_i-s_i) X_i \rv_2\leq \E \sup_{t\in T_k} \left|\sum_i (t_i-z_i)X_i \right|+C(d)\norxp{(t_i-s_i)_i}{\ln(m)}. \label{polo1}
\end{align}
The last inequality is true since (we recall $\ln(m)\geq \ln(8) >2$)
% (in the first inequality we use \eqref{momsum})
$$\lv \sum_i (z_i-s_i) X_i \rv_2 \leq \lv \sum_i (z_i-s_i) X_i \rv_{\ln(m)} \leq C(d) \norxp{(t_i-s_i)_i}{\ln(m)}.$$
Let us also notice that
\begin{align*}
 &\E \sup_{t\in T_k} \left|\sum_i (t_i-z_i)X_i \right|=\E \max \left( \left(\sup_{t\in T_k} \sum_i (t_i-z_i)X_i\right)_{+}, \left(\sup_{t\in T_k} \sum_i (t_i-z_i)X_i\right)_{-} \right) \\
&\leq \E  \left(\sup_{t\in T_k} \sum_i (t_i-z_i)X_i\right)_{+} + \E \left(\sup_{t\in T_k} \sum_i (t_i-z_i)X_i\right)_{-} \\
&=2\E \sup_{t\in T_k} \left( \sum_i (t_i-z_i)X_i \right)_{+}=2\E \sup_{t\in T_k}  \sum_i (t_i-z_i)X_i, 
\end{align*}
where in the second equality we used that $X_i$ are symmetric and in the last one that $z \in T_k$. The above together with \eqref{polo1} imply \eqref{polo2}.
\end{proof}

The next Theorem (together with Lemma \ref{11}) allows us to pass from the bounds on expectations of suprema of
Gaussian processes developed in \cite{AdLat} (Theorem \ref{A.13}) to empirical processes involving general random variables with bounded fourth
moments (in particular all random variables from the $\Sd(d)$ class). 

\begin{twr} \label{17}
Let $p\geq 1$ and $T\subset B^n_2+\sqrt{p}B_1$. There is a decomposition $T=\bigcup_{l=1}^N T_l$, $N\leq e^{Cp}$ such that for every $l \leq N$ and $z \in \R^n$ the following holds:
\begin{equation}
\E \sup_{x \in T_l} \sum_{i,j} \aaa x_i g_j z_j \leq C \left(\sum_{i,j} \aaa^2 z_j^4 \right)^{\frac{1}{4}} \left(\sum_{i,j} \aaa^2 \right)^{\frac{1}{4}}. \label{ww.1}
\end{equation}

\end{twr}

\begin{proof}
Let $\al_z(x)=\left( \sum_j z_j^2 \left(\sum_i \aaa x_i \right)^2 \right)^{\frac{1}{2}}$. By the Cauchy-Schwarz inequality,
\begin{equation}
\al_z(x)\leq \left( \sum_{i,j} z_j^4 \aaa^2 \right)^{\frac{1}{4}}\beta(x), \label{17.1}\end{equation}
where 
$$\beta(x)=\left(\sum_j \frac{\left(\sum_i \aaa x_i \right)^4}{\sum_i \aaa^2} \right)^{\frac{1}{4}}.$$
Let $\wy=(\wy_j)_j$, where $\wy_j$ are i.i.d symmetric exponential r.v's with the density $e^{-|x|}/2 $. We have

\begin{equation}
\E\beta(\wy)\leq \left(\E\beta(\wy)^4\right)^{\frac{1}{4}} \leq C \left(\sum_{i,j} \aaa^2 \right)^{\frac{1}{4}}. \label{loc1}
\end{equation}

Using Corollary \ref{A.12} with, $a=\sqrt{p}$, $t=\frac{1}{\sqrt{p}}$ and $\rho_{\al}(x,y)=\beta(x-y)$ we can decompose $T$ into $\bigcup_{l=1}^N T_l$ in such a way that $N\leq \exp(Cp)$ and
\begin{align}
\forall_{l \leq N} \sup_{x, \widetilde{x} \in T_l} \beta(x-\widetilde{x}) \leq \frac{C}{\sqrt{p}} \left( \sum_{i,j} \aaa^2 \right)^{\frac{1}{4}}. \label{gwiazdaa}
\end{align}
By Lemma \ref{A.13} we obtain
\begin{align*}
&\frac{1}{C} \E \sup_{x \in T_l} \sum_{i,j} \aaa x_i g_j z_j \leq \sqrt{\sum_{i,j} \aaa^2 z^2_j} + \sqrt{p} \sup_{x,\widetilde{x} \in T_l} \sqrt{\sum_j z_j^2 \left(\sum_i \aaa (x_i-\widetilde{x}_i) \right)^2} \\
&\leq \sqrt{\sum_{i,j} \aaa^2 z^2_j} + \sqrt{p} \sup_{x,\widetilde{x} \in T_l} \left(\sum_{i,j} \aaa^2 z^4_j \right)^{\frac{1}{4}}\beta(x-\widetilde{x}) \leq C\left(\sum_{i,j} \aaa^2 z^4_j\right)^{\frac{1}{4}}\left(\sum_{i,j} \aaa^2 \right)^{\frac{1}{4}},
\end{align*}
where in the last inequality we used the Cauchy-Schwarz inequality and \eqref{gwiazdaa}.

\end{proof}

\begin{fak}\label{19}
For any symmetric set $T\subset\sqrt{p}B^n_2+pB^n_1$,  we have
$$S(T):=\E \sup_{x \in T} \sum_{i,j} \aaa x_i Y_j \leq C(d) \left(\sup_{x \in T} \nory{\left(\sum_i \aaa x_i \right)_j}+\nory{\left(\sqrt{\sum_i \aaa^2}\right)_j}\right).$$

\end{fak}

\begin{proof}
Obviously,
\begin{equation}
S(T)\leq \E \sup_{x \in T}  \sum_{i}\sum_{j\leq p} \aaa x_i Y_j+\E \sup_{x \in T}  \sum_{i}\sum_{j>p} \aaa x_i Y_j=:S_1(T)+S_2(T). \label{19.1}
\end{equation}
By Fact \ref{prockrt} and Theorem \ref{jeden} we have
\begin{equation}
S_1(T)\leq C \sup_{x \in T} \lv \sum_{i,j} \aaa x_i Y_j \rv_p  \leq C(d) \sup_{x \in T} \nory{\left(\sum_i \aaa x_i \right)_j}.
\end{equation}

Now we bound $S_2(T)$. By Theorem \ref{17} We may decompose $T$ into $T=\bigcup_{l=1}^N T_l$ in such a way that $N\leq \exp(Cp)$ and \eqref{ww.1} holds for $T_l/\sqrt{p}$ instead of $T_l$ and $(\aaa)_{i\geq 1,j>p}$ instead of $(\aaa)_{i,j}$. Using Lemmas \ref{11} and \ref{gwiazda} we get

\begin{align}
S_2(T)\leq C(d)  \left(  \max_{l \leq N} \E \sup_{x \in T_l} \sum_{i}\sum_{j>p} \aaa x_i Y_j +\sup_{x \in T} \nory{\left(\sum_i \aaa x_i \right)_j} \right) \label{19.3}.
\end{align}

By the symmetry of $Y_j$'s, Jensen's inequality and the contraction principle
\begin{align*}
   \E \sup_{x \in T_l} \sum_{i}\sum_{j>p} \aaa x_i Y_j=    \E \sup_{x \in T_l} \sum_{i}\sum_{j>p} \aaa x_i \eps_j |Y_j| \leq \sqrt{\frac{2p}{\pi}}   \E \sup_{x \in T_l/ \sqrt{p}}  \sum_{i}\sum_{j>p} \aaa x_i g_j |Y_j|.
\end{align*}
 Theorem 
\ref{17} states that last term in the above formula does not exceed
\begin{align}
C(d) \sqrt{p}\ \E\left(  \sum_{i}\sum_{j>p} \aaa^2 Y^4_j \right)^{\frac{1}{4}}\left(  \sum_{i}\sum_{j>p} \aaa^2 \right)^{\frac{1}{4}}  &\leq C(d) \sqrt{p} \left(\sum_i \sum_{j \geq p} \aaa^2 \E Y_j^4 \right)^\frac{1}{4}  \nonumber \\
&\leq C(d) \sqrt{p} \sqrt{ \sum_{i}\sum_{j>p} \aaa^2}  \label{19.4}.
\end{align}

In the last inequality we used  $\lv Y_j \rv_4 \leq 2^d \lv Y_j \rv_2=2^d/e$. From Lemma \ref{rad} and \eqref{monot}
\begin{align}
\frac{\sqrt{p}}{2} \sqrt{ \sum_{i}\sum_{j>p} \aaa^2} \leq  \sup \left\{\sum_j \sqrt{\sum_i \aaa^2}t_j \ \Big{|} \sum_j t^2_j \leq p,\ \forall_{j\leq n} |t_j|\leq 1 \right\}\leq  \nory{\left(\sqrt{\sum_i \aaa^2} \right)_j}. \label{19.5}
\end{align}

The assertion follows by \eqref{19.1}-\eqref{19.5}.

\end{proof}

\section{Proof of Theorem \ref{1}} \label{sekcja5}
%The plan of the proof is as follows. First we will prove the lower bound. Then we will estimate $\lv \sum_{i,j} \aaa X_i Y_j\rv_p$  under stronger assumptions that all $(Y_j)_j$ are in $\Sd(d)$ class and $X_i$ having log-concave tails. FiFinallye prove Theorem \ref{1} in full generality.

We are ready to prove the main theorem of this paper. We begin with the lower bound.

%\subsection{Lower bound}
Repeated application of \eqref{lepwzor} gives
\begin{equation}
\lv \sum_{i,j} \aaa X_i Y_j \rv_p \geq c(d) \norxy{\aaa}. \label{6.1.1}
\end{equation}
Symmetry of $Y_j's$, Jensen's inequality, \eqref{momsum} with $p=1$, normalization $\lv X_i \rv_2=1/e$ and  \eqref{lepwzor} imply
\begin{align}
\lv \sum_{i,j} \aaa X_i Y_j \rv_p &= \lv \sum_j Y_j   \left| \sum_i \aaa X_i \right|\rv_p
\geq \lv \sum_j Y_j \E_X \left|\sum_i \aaa X_i  \right| \rv_p \nonumber \\
&\geq \frac{1}{C(d)} \lv \sum_j Y_j \sqrt{\E_X \left(\sum_i \aaa X_i \right)^2} \rv_p  \geq \frac{1}{C(d)}\nory{\left(\sum_i \aaa^2 \right)_j}. \label{6.1.2}
\end{align}
In the same way we show 
\begin{align}
\lv \sum_{i,j} \aaa X_i Y_j \rv_p \geq \frac{1}{C(d)} \lv \sqrt{\sum_j \aaa^2} \rv_{X,p}. \label{6.1.2.11}
\end{align}
Inequalities \eqref{6.1.1} and \eqref{6.1.2.11}  gives the lower bound in Theorem \ref{1}.

Now we establish the upper bound. To this end we observe that Theorems \ref{jeden} and \ref{9}  and Lemma \ref{porsum}  yield
\begin{align}
\lv \sum_{i,j} \aaa X_i Y_j \rv_p \leq C(d) \left( \E \robnorm{\left( \sum_j \aaa Y_j \right)_i}{d} + \norxy{(\aaa)_{i,j}} \right), \nonumber
 \end{align}
whereas $ \E \robnorm{\left( \sum_j \aaa Y_j \right)_i}{d}$ is bounded by the following Proposition.
%\ref{naprawa}  
%implies
%\begin{align}
%\E \robnorm{\left( \sum_j \aaa Y_j \right)_i}{d} \leq C(d) \left( \norxy{(\aaa)_{i,j}}+\nory{\left(\sqrt{\sum_i \aaa^2}\right)_j }\right). \nonumber
%\end{align}

\begin{prep}\label{naprawa}
For any $d \geq 1$ we have
\begin{align}
\E \robnorm{\left(\sum_j \aaa Y_j \right)_i}{d} \leq C(d) \left(  \norxy{(\aaa)}+ \nory{\left(\sqrt{\sum_i \aaa^2} \right)_j} \right). \label{teza}
\end{align}
\end{prep}
\begin{proof}

Since the functions $\mx(t)$ are symmetric and $1+x\leq 2+x\1_{\{x \geq 1 \}}$ we have
\begin{align*}
\E \robnorm{\left(\sum_j \aaa Y_j \right)_i}{d}&=\E \sup \left\{ \sum_{i}   x^1_i \prod_{k=2}^{d} (1+x^k_i) \left| \sum_j \aaa Y_j \right| \ \Big{|} \ \forall_{k\leq d} \sum_i \mx(x^k_i)\leq p \right\} \\
&\leq 2^{d-1} \Bigg{[} \E \sup \left\{ \sum_{i}   x^1_i  \left| \sum_j \aaa Y_j \right| \ \Big{|} \  \sum_i \mx(x^1_i)\leq p \right\}\\
&\ \ \ +\sum_{I \subset [d]} \E \sup \left\{ \sum_{i}   \prod_{k\in I} x^k_i \1_{\{|x^k_i| \geq 1\}} \left| \sum_j \aaa Y_j \right| \ \Big{|} \ \forall_{k \in I} \sum_i \mx(x^k_i)\leq p \right\} \Bigg{]}.\\
\end{align*}

Since $\{x \in \R^n \ | \ \sum_i \mx(x_i)\leq p \}\subset \sqrt{p}B^n_2+pB^n_1$ (recall \eqref{ulubionaetykieta}) Lemma \ref{19} implies 
\begin{align*}
&\E \sup \left\{ \sum_{i}   x^1_i  \left| \sum_j \aaa Y_j \right| \ \Big{|} \  \sum_i \mx(x^1_i)\leq p \right\}=\E \sup \left\{ \sum_{i}   x^1_i   \sum_j \aaa Y_j  \ \Big{|} \  \sum_i \mx(x^1_i)\leq p \right\} \\
&\leq C(d) \left( \sup \left\{  \nory{\left(\sum_i \aaa x^1_i\right)_j}   \ \Big{|} \  \sum_i \mx(x^1_i)  \leq p\right\}+\nory{\left(\sqrt{\sum_i \aaa^2}\right)_j}\right) \\
&\leq C(d) \left(   \norxy{(a_i)} + \nory{\left(\sqrt{\sum_i \aaa^2} \right)_j} \right).
\end{align*}
If $I \subset [d]$ then
\begin{align*}
&\E \sup \left\{ \sum_{i}   \prod_{k\in I} x^k_i \1_{\{|x^k_i| \geq 1\}} \left| \sum_j \aaa Y_j \right| \ \Big{|} \ \forall_{k \in I} \sum_i \mx(x^k_i)\leq p \right\}\\
&\leq \E \sup \left\{ \sum_{i}   \prod_{k=1}^d x^k_i \1_{\{|x^k_i| \geq 1\}} \left| \sum_j \aaa Y_j \right| \ \Big{|} \ \forall_{k \in I} \sum_i \mx(x^k_i)\leq p \right\}=:S,
\end{align*}
where in the inequality we choose $x^k_i=x^{k_0}_i$ for any $k \notin I$ and some $k_0 \in I$.

Let  $T,U,V$ be the sets defined in \eqref{T}-\eqref{V}. Then $T \subset U+V$. So we have that
\begin{align}
S&\leq \E \sup_{x^1,\ldots,x^d \in U} \sum_{i} \prod_{k=1}^d x^k_i \left| \sum_j \aaa Y_j \right| +\sum_{\stackrel{I\subset [d]}{I\neq \emptyset}} \E \sup_{x^i \in U \textrm{ for } i\in I }\ \  \sup_{x^i \in V \textrm{ for } i\notin I} \sum_{i} \prod_{k=1}^d x^k_i \left| \sum_j \aaa Y_j \right| \nonumber \\
&\leq  \E \sup_{x^1,\ldots,x^d \in U} \sum_{i} \prod_{k=1}^d x^k_i \left| \sum_j \aaa Y_j \right| +\left(2^d-1\right) \E \sup_{x^1,\ldots,x^{d-1}\in T,x^d\in V } \sum_{i} \prod_{k=1}^d x^k_i \left| \sum_j \aaa Y_j \right|.  \label{przystanek}
\end{align}
The last inequality follows by the symmetry and the inclusions $U,V\subset T$. Observe that \eqref{poU} implies 
$$\E \sup_{x^1,\ldots,x^d \in U} \sum_{i} \prod_{k=1}^d x^k_i \left| \sum_j \aaa Y_j \right| \leq C(d) \nory{\left(\sqrt{\sum_i \aaa^2} \right)_j},  $$
whereas \eqref{poV} implies
$$ \E \sup_{x^1,\ldots,x^{d-1}\in T,x^d\in V } \sum_{i} \prod_{k=1}^d x^k_i \left| \sum_j \aaa Y_j \right| \leq C(d)   \norxy{(a_i)}.$$

\end{proof}

\appendix

\section{Appendix}
In this section we gather results from previous work used in the paper as well as proofs of several simple technical results. We use notation introduced in Sections \ref{sekcja2} and \ref{prel}.

%We apply the convention that whenever $X_1,X_2,\ldots;\ Y_1,Y_2,\ldots$ are random variables, then $N^X_i(t)= -\ln\Pro(|X_i|\geq t)$,

%$$\nx(t)=\begin{cases} t^2 &\textrm{for } |t|\leq 1 \\ N^X_i(t) &\textrm{otherwise.} \end{cases}$$
%and we define $N^Y_j(t),\ny(t)$ in the same manner.
\begin{lem}\label{spraw}
Assume that $T\subset \R^n$ is bounded and $span(T)=\R^n$. Then $\lv x \rv = \sup_{t \in T} \left| \sum_i x_i t_i \right|$ is a norm. In particular $\norx{(a_i)},\ \norxy{(\aaa)}$ are norms.
\end{lem}
\begin{proof}
It is clear that $\lv \ \cdot \ \rv$  is well-defined, homogeneous and satisfies the triangle inequality. Now observe that $\lv x \rv=0$  gives $ \sum_i x_i t_i=0$ for all $t \in span(T)=\R^n$, hence $x=0$ and the first part of the assertion follows.
Now let 
$$\hat{T}=\left\{ (x_iy_j)_{i,j} \in \R^{n^2} \ \Big{|} \ \sum_i \nx(x_i) \leq p,\ \sum_j \ny(y_j)\leq p \right\}.$$

Observe that $\hat{T}$ spans $\R^{n^2}$ and is bounded since for any r.v $X$, $-\ln \Pro (|X|\geq t)\rightarrow \infty$ as $t\rightarrow \infty$. Analogously  we prove that $\norx{ \ \cdot \ }$ is a norm.
\end{proof}

\begin{fak} \label{prockrt}
For any $p\geq 1$, any set $T \subset \R^{\lfloor p \rfloor}$ which fulfills assumptions of Lemma \ref{spraw} and any r.v's $Z_1,\ldots,Z_p$ we have
$$\E \sup_{x \in T} \lm \sum_{i\leq p} x_i Z_i \rmo \leq  \left( \E \sup_{x \in T} \lm \sum_{i\leq p} x_i Z_i \rmo^p \right)^{\frac{1}{p}} \leq 10 \sup_{x \in T} \lv \sum_i x_i Z_i \rv_p.$$
\end{fak}

\begin{proof}
 Consider a norm on $\R^{\lfloor p \rfloor }$ given by 
$$\lv (y_1,\ldots,y_{\p}) \rv =\sup_{x\in T} \lm \sum_{i \leq p} x_i y_i \rmo.  $$
Let $K$ be the unit ball of the dual norm $\lv \ \rv_*$. Then $K= \mathrm{conv}(T \cup -T)$. Let $M$ be a $1/2$ net in $K$ (with respect to $\lv \ \rv_*$) of cardinality not larger than $5^\p$ ($M$ exists by standard volumetric arguments). Then for all $y \in \R^\p$
$$ \lv y \rv \leq 2 \sup_{u \in M} \lm \sum_{i \leq p} u_i y_i \rmo.$$
Thus
$$\E \sup_{x \in T} \lm \sum_{i\leq p} x_i Z_i \rmo^p \leq 2^p \E \sup_{u \in M} \lm \sum_{i\leq p } u_i Z_i \rmo^p. $$
Observe that
\begin{align*}
\left( \E \sup_{u \in M} \lm \sum_{i\leq p } u_i Z_i \rmo^p \right)^{\frac{1}{p}} \leq \left(\sum_{u \in M} \E \left| \sum_{i\leq p } u_i Z_i \right|^p     \right)^{\frac{1}{p}} \leq \left(5^p \sup_{u \in M} \E \left| \sum_{i\leq p } u_i Z_i \right|^p \right)^\frac{1}{p}\leq 5 \sup_{x \in T} \lv \sum x_i Z_i \rv_p.
\end{align*}

\end{proof}

\begin{lem} \label{rad}
For any $p \geq 1$ and any $a_1\geq a_2 \geq \ldots \geq a_n \geq 0$,
$$\frac{1}{2}  \left(\sum_{i \leq p} a_i + \sqrt{p} \sqrt{\sum_{i>p} a^2_i} \right) \leq  \sup \left\{\sum_i a_i t_i \ \Big{|} \ \sum_i t^2_i \leq p,\ |t_i|\leq 1 \right\} \leq \sum_{i \leq p} a_i + \sqrt{p} \sqrt{\sum_{i>p} a^2_i}.$$

\end{lem}
\begin{proof}
Denote $M=\sup \left\{\sum_i a_i t_i \ \Big{|} \ \sum_i t^2_i \leq p,\ |t_i|\leq 1 \right\}$. By choosing $t_i=1$ for $i \leq p$ and $t_i=0$ for $i>p$ we see that $M\geq \sum_{i\leq p } a_i$. 

Now let $k=\lfloor p \rfloor +1$, $A=(ka^2_k+\sum_{i>k} a^2_i)^{\frac{1}{2}}$, $t_i=\sqrt{p}a_k/A$ for $i\leq k$ and $t_i=\sqrt{p}a_i/a$ for $i>k$. Then $|t_i|\leq 1$, $\sum t^2_i=p$ and
$$M\geq \sum_i a_i t_i\geq \sqrt{p} A \geq \sqrt{p} \sqrt{\sum_{i>p} a^2_i}.$$
To show the upper bound it is enough to observe that $\sum a_i t_i \leq \lv t \rv_{\infty}\sum_{i\leq p} a_i+\lv t \rv_2 \sqrt{\sum_{i>p} a^2_i}$.
\end{proof}

\begin{lem}\label{osz}
Let $r \leq 1 <R$ and
$$f(t)=\begin{cases}t^2 &\textrm{for } |t|\leq 1 \\ |t|^r &\textrm{for } 1<|t|\leq R \\ \infty &\textrm{for } R<|t|.   \end{cases}$$
Then for any $p\geq 1$ and  $a_1\geq a_2 \geq \ldots \geq a_n \geq 0$
$$\sup \left\{\sum_i a_i t_i \ \Big{|} \ \sum_i f(t_i)\leq p \right\}\sim \begin{cases}  \sqrt{p} \sqrt{\sum_i a^2_i}+p^{\frac{1}{r}}a_1 &\textrm{for } 1\leq p \leq R^r \\ \sqrt{p} \sqrt{\sum_i a^2_i}+ R \sum_{i\leq \frac{p}{R^r}} a_i &\textrm{for } R^r<p \leq R^2 \\\sqrt{p} \sqrt{\sum_{i\geq \frac{p}{R^2}} a^2_i }+ R \sum_{i\leq \frac{p}{R^r}} a_i   &\textrm{for } R^2<p. \end{cases} $$ 
\end{lem}

\begin{proof}
Since $|t|^r\leq t^2$ for $|t|\geq 1$ and $|t|^r \geq t^2$ for $|t| \leq 1$ we obtain
\begin{align}
&\frac{1}{2} \left( \sup \left\{\sum_i a_i t_i \ \Big{|} \ \sum_i t_i^2 \leq p,\ |t_i|\leq R \right\}+\sup \left\{\sum_i a_i t_i \ \Big{|} \ \sum_i |t_i|^r\leq p,\ |t_i|\leq R \right\} \right) \nonumber  \\
 &\leq \sup \left\{\sum_i a_i t_i \ \Big{|} \ \sum_i f(t_i)\leq p \right\} \nonumber \\
&\leq \sup \left\{\sum_i a_i t_i \ \Big{|} \ \sum_i t_i^2 \leq p,\ |t_i|\leq R \right\}+\sup \left\{\sum_i a_i t_i \ \Big{|} \ \sum_i |t_i|^r\leq p,\ |t_i|\leq R \right\}=:S_1+S_2. \label{por}
\end{align}
We will estimate $S_1$ and $S_2$ separately.
\begin{enumerate}
\item Obviously $S_1=\sqrt{p} \sqrt{\sum_i a^2_i}$ for $p\leq R^2$. Assume that $p>R^2$. By homogeneity and Lemma \ref{rad}
$$S_1=R \left\{\sum_i a_i t_i \ \Big{|} \ \sum_i t^2_i \leq \frac{p}{R^2},\ |t_i|\leq 1   \right\} \sim  R \sum_{i\leq \frac{p}{R^2}} a_i + \sqrt{p} \sqrt{\sum_{i\geq \frac{p}{R^2}} a^2_i }.$$
\item It is easy to see that $S_2=p^{\frac{1}{r}}a_1=p^{\frac{1}{r}}\lv (a_i) \rv_\infty$ for $p\leq R^r$. Assume that $p>R^r$. Because $0<r\leq 1$, by homogeneity
\begin{align*}
S_2&=R \left\{\sum_i a_i t_i \ \Big{|} \ \sum_i |t_i|^r \leq \frac{p}{R^r},\ |t_i|\leq 1   \right\}\leq R \left\{\sum_i a_i t_i \ \Big{|} \ \sum_i |t_i| \leq \frac{p}{R^r},\ |t_i|\leq 1   \right\}\\
&\leq 2R \sum_{i\leq  \frac{p}{R^r}  } a_i
\end{align*} 
Moreover by picking $t_1=\ldots=t_{\lfloor \frac{p}{R^r}  \rfloor}=1$ and $t_{\lfloor \frac{p}{R^r}  \rfloor+1}=\ldots=t_n=0$ we observe that 
$$S_2 \geq  R \sum_{i\leq  \frac{p}{R^r}  } a_i.$$
\end{enumerate}
Since $R^2\geq R^r$ the asertion easily follows.

\end{proof}

 Theorem  \ref{gl} was formulated in \cite{GluKw} in a slightly different manner. The below formulation can be found for instance in \cite{Some} (Theorem 2 therein).
\begin{twr}[Gluskin-Kwapie\'n estimate]\label{gl}
Let $X_1,\ldots,X_n $ be independent, symmetric r.v's with log-concave tails which fulfills following normalization condition
\begin{equation}
\forall_{i\leq n} \ \inf \{ t>0 : N^X_i(t)\geq 1 \}=1. \label{nor}
\end{equation}
 Then for any $p \geq 1, \ a_1,\ldots a_n \in \R$  we have
$$ \lv \sum_i a_i X_i \rv \sim \norx{(a_i)}.   $$

\end{twr}

\begin{lem} \cite[Lemma 3.2]{Latloch} \label{A.6}
Assume that $X_1,\ldots,X_n$ satisfy assumptions of Theorem \ref{gl}. Let $T$ be a finite subset of $(\R_+)^n$. Then for any $p\geq 1$ we have
$$\lv \max_{(t_1,\ldots,t_n) \in T} \max\left(\sum_i t_i X_i-C\sum_i t_i, 0\right) \rv_p \leq C\left(p+\ln |T| \right)\max_{(t_1,\ldots,t_n) \in A} \max_i t_i.$$

\end{lem}

\begin{cor}\cite[Corollary 3]{Some}\label{A.7}
Assume that $X_1,\ldots,X_n$ and $Y_1,\ldots,Y_n$ satisfy the assumptions of Theorem \ref{gl}. Then for any $p\geq 1$
\begin{align*}
%\E \sup \left\{\sum_{i,j} a_{i,j} X_i y_j \ \Big{|}\ \sum_j \ny(y_j) \leq p \right\}=
\E \nory{\left(\sum_i a_{i,j} X_i \right)_j}\leq C \left(\norxy{(a_{i,j})} + \nory{\left( \sqrt{\sum_i a^2_{i,j} } \right)_j}  \right).
\end{align*}
\end{cor}

\begin{twr}[Kwapie\'n decoupling inequalities {\cite[ Theorem 2]{Kwa}} ]\label{A.8}
Let $F$ be a Banach space and\\
 $(X^i_j)_{i\leq d,j \leq n},(X_j)_{j\leq n}$ are  independent symmetric random variables such that
$$\forall_{j \leq n}\ X^1_j,X^2_j,\ldots,X^d_j,X_j \textrm{ are i.i.d}.$$
Assume that array $(\ai)_{\id} \in F^{n^d}$ is tetrahedral ($i_k = i_l$ for $k\neq l$, $k,l \leq n$ implies $\ai = 0$) and symmetric ($a_{\id}=a_{i_{\pi(1)},\ldots,i_{\pi(d)}}$ for all permutations $\pi$ of $[d]$). Then for any convex function $\phi:F \rightarrow \R_+$
\begin{align*}
\E \phi \left(\frac{1}{C} \sum_\id \ai \il  \right) \leq& \E \phi \left(\sum_\id \ai X_{i_1} \cdot \ldots \cdot X_{i_d} \right) \\
 &\leq \E \phi \left(C \sum_\id \ai \il  \right).
\end{align*}

\end{twr}

\begin{twr}[V. H. de la Pe{\~n}a and S. J. Montgomery-Smith decoupling bounds, {\cite[Theorem 1]{Pen}}.] \label{A.9}
Under the assumptions of Theorem \ref{A.8} we have for any $t\geq 0$,
\begin{align*}
\frac{1}{C} \Pro \left( \lv  \sum_\id \ai \il \rv \geq Ct \right) \leq& \Pro \left( \lv \sum_\id \ai X_{i_1} \cdot \ldots \cdot X_{i_d} \rv \geq t \right)\\  &\leq C \Pro \left( \lv  \sum_\id \ai \il \rv \geq \frac{t}{C} \right).  
\end{align*}
\end{twr}

\begin{twr}[Lata{\l}a bound on $L_p$-norms of linear forms \cite{Chaos1}]\label{A.10}
Let $X_1,X_2,\ldots,X_n$ be independent nonnegative random variables or independent symmetric random variables. Then
$$\lv \sum_i X_i \rv_p \sim \inf \left\{ t>0 :\ \sum_i \ln \left( \E \left(1+\frac{X_i}{t} \right)^p  \right) \leq p \right\}.$$
\end{twr}

\begin{rem}\label{kryt}\cite[Remark 3.2]{Mel}
Let $X$ be a symmetric random variable, such that $\E X=1$ and there exist constants $K,\beta >0$ such that $\ln \Pro \left(|X|\geq Ktx \right)\leq t^\beta \ln \Pro \left(|X|\geq x \right)$ for any $t,x\geq 1$. Then $X$ fulfills the moment condition \eqref{mc} with $\al=\al(K,\beta).$

\end{rem}

%Ponizej wniosek 5.7 w pelnej ogolnosci

%\begin{cor}[Corollary $5.7$ from \cite{AdLat}]
%Let $(\mathcal{E}^i_j)_{i,j \in \N}$ be an array of i.i.d, symmetric exponential variables with the density $\exp(-|x|/2)$, $\al(\ \cdot \ )$ be an arbitrary norm on $\R^{n^d}$ and $\rho_\al$ be a distance on $\R^{n \cdot d}$ defined by
%$$\rho_\al\left((x^1,\ldots,x^d),(y^1,\ldots,y^d) \right)=\al \left( \left( \prod_{j=1}^d x^j_{i_j}-\prod_{j=1}^d y^j_{i_j} \right)_{1\leq i_1,\ldots,i_d\leq n}  \right).$$
%Then for any $a \in \R_+$, $T \subset (B^n_2+aB^n_1)^d$, $t \in (0,1]$,
%$$N(T,\rho_\al,V^T_d(\al,t))\leq \exp\left(C(d) (t^{-2}+at^{-1} ) \right), $$
%where
%$$V^T_d(\al,t)=\sum_{k=1}^d t^k \sum_{I \subset [d]:\ |I|=k} \sup_{(x^1,\ldots,x^d) \in T} \E \al \left( \left(\prod_{k \notin I} x^k_{i_k} \prod_{k \in I} \mathcal{E}^k_{i_k} \right)_{1\leq i_1,\ldots,i_d\leq n} \right).$$

%\end{cor}

\begin{cor}[Corollary $5.7$ from \cite{AdLat} with $d=1$]\label{A.12}
Let $\mathcal{E}_1,\ldots, \mathcal{E}_n$ be i.i.d random variables with the density $\exp(-|t|/2)$, $\al( \ \cdot \ )$ be a norm on $\R^n$ and $\rho_\al$ be a distance on $\R^n$ defined by 
$$\rho_\al(x,y)=\al(x-y).$$
Then for any $a \in \R_+$, $T \subset B^n_2+aB^n_1$, $t \in (0,1]$,
$$N(T,\rho_\al,Ct \E \al(\mathcal{E}_1,\ldots,\mathcal{E}_n ))\leq \exp\left( t^{-2}+at^{-1}  \right).$$ 

\end{cor}

\begin{lem}\cite[Lemma 6.3]{AdLat}\label{A.13}
For any matrix $B=(b_{i,j})_{i,j\leq n}$, any $T\subset B^n_2+aB^n_1$ and $a\geq 1$,
$$\E \sup_{x \in T} \sum_{i,j} b_{i,j} x_i g_j \leq C \left(\sqrt{\sum_{i,j} b^2_{i,j}}+a \cdot \sup_{x,x' \in T} \sqrt{\sum_j \left(\sum_i b_{i,j} (x_i-x'_i) \right)^2} \right).$$
\end{lem}

%\begin{twr}\label{A.14}
%Assume that $X_1,\ldots,X_n$ are independent random variables  such that
%$$\forall_{i \leq n} \ \forall_{p\geq q \geq 2} \lv X_i \rv_p \leq \al \frac{p}{q} \lv X_i \rv_p$$
%for some constant $\al >0$.
%Then for any $T \subset \R^n$ we have
%$$ \lv \sup_{t \in T} \sum_i t_i X_i \rv_p \leq C(\al) \left( \E \sup_{t \in T} %\left| \sum_i t_i X_i \right| + \sup_{t \in T} \lv \sum_i t_i X \rv_p \right).$$
%\end{twr}

%\begin{twr}[Theorem 1 from \cite{Latban}]
%Let $F$ be a Banach space with a norm $\al(\ \cdot \ )$. Assume that $X_1,\ldots,X_n$ are independent random variables with log-concave tails. Then for any $p \geq 1$, $a_1,\ldots,a_n \in F$ 
%$$ \lv \al\left(\sum_i a_i X_i\right) \rv_p \leq C\left(\E\al\left(\sum_i a_i X_i\right)+\sup_{\varphi \in F*,\ \lv \varphi \rv \leq 1} \lv \varphi \left(\sum_i a_i X_i \right) \rv_p  \right) $$

%\end{twr}

\begin{twr}\cite[Theorem 1.1]{LatSt} \label{9}
Assume that $X_1,\ldots,X_n$ are independent r.v's such that 
$$\forall_i \forall_{p\geq 1} \lv X_i \rv_{2p} \leq \al \lv X_i \rv_p \textrm{ for some } \al>0.$$
Then for any non empty set $T\subset \R^n$ and $p \geq 1$ we have
$$\momp{\sup_{t \in T} \left|\sum_i t_i X_i \right| } \leq C(d) \left( \lv \sup_{t \in T} \left|\sum_i t_i X_i\right| \rv_1 + \sup_{t \in T} \lv \sum_i t_i X_i \rv_p \right).$$
\end{twr}

\textbf{Acknowledgements:} I would like to express my gratitude to my wife, for supporting me. I would also like to thank prof. R. Lata{\l{}}a for a significant improvement of the paper.

\noindent
Rafa{\l} Meller\\
Institute of Mathematics\\
University of Warsaw\\
Banacha 2, 02-097 Warszawa, Poland\\
{\tt r.meller@mimuw.edu.pl}

\end{document}